\documentclass[a4paper,reqno,10pt]{amsart}

\usepackage{epsf,graphicx,epsfig}
\usepackage{amscd}
\usepackage{amsmath,latexsym,amssymb,amsthm}
\usepackage[nospace,noadjust]{cite}
\usepackage{textcomp}
\usepackage{dsfont}
\usepackage{setspace,cite}
\usepackage{lscape,fancyhdr,fancybox}
\usepackage{stmaryrd}
\usepackage[all,cmtip]{xy}
\usepackage{tikz}
\usepackage{cancel}
\usetikzlibrary{shapes,arrows,decorations.markings}
\usepackage{graphicx}

\usepackage{color}
\usepackage{url}
\usepackage{enumerate}
\usepackage[mathscr]{euscript}

  \theoremstyle{plain}

\swapnumbers
    \newtheorem{thm}{Theorem}[section]
    \newtheorem{prop}[thm]{Proposition}
     
   \newtheorem{lemma}[thm]{Lemma}

    \newtheorem{subsec}[thm]{}
\theoremstyle{definition}
    \newtheorem{defn}[thm]{Definition}
        \newtheorem{remark}[thm]{Remark}
    \newtheorem{exam}[thm]{Example}

\theoremstyle{remark}

\renewcommand{\Im}{\operatorname{Im}}

\newcommand{\Der}{\operatorname{Der}}

\newcommand{\Hom}{\operatorname{Hom}}
\newcommand{\Aut}{\operatorname{Aut}}
\newcommand{\Ker}{\operatorname{Ker}}
\newcommand{\Lie}{\operatorname{Lie}}
\newcommand{\ALie}{\operatorname{ALie}}
\newcommand{\Leib}{\operatorname{Leib}}
\newcommand{\ab}{\operatorname{ab}}
\newcommand{\nab}{\operatorname{nab}}
\newcommand{\End}{\operatorname{End}}
\newcommand{\Ext}{\operatorname{Ext}}

\newcommand{\Id}{\operatorname{Id}}

\title{}
\author{}
\date{}
\usepackage{amssymb}

\usepackage{hyperref}
\hypersetup{
	colorlinks,
	citecolor=blue,
	filecolor=black,
	linkcolor=blue,
	urlcolor=black
}

\begin{document}

\title{$2$-term averaging $L_\infty$-algebras and non-abelian extensions of averaging Lie algebras}

\author{Apurba Das}
\address{Department of Mathematics,
Indian Institute of Technology, Kharagpur 721302, West Bengal, India.}
\email{apurbadas348@gmail.com, apurbadas348@maths.iitkgp.ac.in}

\author{Sourav Sen}
\address{Tata Institute of Fundamental Research, Mumbai 400005, India.}
\email{sourav.sen3@gmail.com}

\begin{abstract}
In recent years, averaging operators on Lie algebras (also called embedding tensors in the physics literature) and associated tensor hierarchies form an efficient tool for constructing supergravity and higher gauge theories. A Lie algebra with an averaging operator is called an averaging Lie algebra. In the present paper, we introduce $2$-term averaging $L_\infty$-algebras and give characterizations of some particular classes of such homotopy algebras. Next, we study non-abelian extensions of an averaging Lie algebra by another averaging Lie algebra. We define the second non-abelian cohomology group to classify the equivalence classes of such non-abelian extensions. Next, given a non-abelian extension of averaging Lie algebras, we show that the obstruction for a pair of averaging Lie algebra automorphisms to be inducible can be seen as the image of a suitable Wells map. Finally, we discuss the Wells short exact sequence in the above context.
\end{abstract}

\maketitle



\medskip

 {2020 MSC classification:} 17B40, 17B55, 17B56, 18G45.

 {Keywords:} Averaging Lie algebras, Homotopy averaging Lie algebras, Non-abelian extensions, Wells exact sequence.



\thispagestyle{empty}

\tableofcontents


\medskip

\section{Introduction}
 Studying algebras equipped with additional structures has been of central importance because of their rich mathematical properties and relevance in various disciplines of mathematics and mathematical physics. 
In the past couple of years, algebras equipped with derivations, Rota-Baxter operators and involutions were studied extensively because of their appearance in many different fields \cite{loday,guo,guo-book,braun}. Averaging operators on associative algebras, another interesting and classical object of interest, drew the attention of the mathematical community, especially in the last couple of decades. Let $A$ be an associative algebra. A linear map $P: A \rightarrow A$ is said to be an {averaging operator} on $A$ if
\begin{align*}
    P(a) P(b) = P (P(a) b ) = P(a P(b)), \text{ for } a, b \in A.
\end{align*}
Kamp\'{e} de F\'{e}riet \cite{kamp} first explicitly defined the notion of averaging operators during the 1930s. Although some applications of averaging operators can be traced back to 1895 when O. Reynolds \cite{reynolds} studied averaging operators in the context of turbulence theory in the disguise of Reynolds operators. In the last century, averaging operators were mostly studied on various function spaces and Banach algebras. In 2000, W. Cao \cite{cao} studied averaging operators from an algebraic point of view while constructing free (commutative) averaging algebras. Subsequently, various mathematical studies of averaging algebras were done \cite{pei-guo,pei-bai-guo-ni,gao-zhang,das, wang-zhou} in connections with combinatorics, number theory, replicators of binary operads, cohomology and deformation theory. Averaging operators can also be defined on Lie algebras (see Definition \ref{aver-on-lie}). Such operators appeared in the work of Kotov and Strobl \cite{kotov} by the name of embedding tensors. More precisely, they observed that averaging operators and associated tensor hierarchies form an effective tool in the constructions of supergravity and higher gauge theories (see also \cite{lavau}). It has been observed that an averaging operator induces a Leibniz algebra structure. A Lie algebra $\mathfrak{g}$ equipped with an averaging operator $P$ is called an averaging Lie algebra, denoted by $\mathfrak{g}_P$. Recently, in \cite{mishra-das-hazra}, the authors considered cohomology and deformation theory of averaging Lie algebras using the derived bracket approach. In the present paper, we consider various questions regarding averaging Lie algebras which can be summarized as follows. 

\subsection{2-term homotopy averaging Lie algebras.} The concept of $L_\infty$-algebras (also called strongly homotopy Lie algebras) plays a prominent role in various contexts of mathematics and mathematical physics \cite{stas,lada-markl}. In their fundamental paper \cite{baez}, Baez and Crans first considered $2$-term $L_\infty$-algebras and their close relationship with categorified Lie algebras. Among others, they considered `skeletal' and `strict' $2$-term $L_\infty$-algebras and gave characterizations of them.

By keeping in mind that $2$-term $L_\infty$-algebras are a homotopy analogue of Lie algebras, it is natural to enquire about homotopy averaging operators. An implicit description of a homotopy averaging operator is given in \cite{sheng-embed}. In this paper, we first give an explicit description of a homotopy averaging operator on a $2$-term $L_\infty$-algebra. However, the complete illustration of a homotopy averaging operator on an arbitrary $L_\infty$-algebra is yet to be found. Following the classical case, we call a $2$-term $L_\infty$-algebra equipped with a homotopy averaging operator by a $2$-term averaging $L_\infty$-algebra. We show that `skeletal' $2$-term averaging $L_\infty$-algebras can be characterized by third cocycles of averaging Lie algebras (cf. Proposition \ref{skeletal-1} and Theorem \ref{skeletal-thm}). Next, we introduce crossed modules of averaging Lie algebras and show that `strict' $2$-term averaging $L_\infty$-algebras are characterized by crossed modules of averaging Lie algebras (cf. Theorem \ref{crossed-11}).

\subsection{Non-abelian extensions of averaging Lie algebras.}
 Extensions (e.g. central extensions, abelian extensions, non-abelian extensions etc.) of some mathematical object are useful to understand the underlying structure \cite{hoch-serre, lyn, serre}. Non-abelian extensions, being the most general among all kinds of extensions, deserve a special mention. The theory of non-abelian extensions was first considered by Eilenberg and Maclane \cite{eilenberg-maclane} for abstract groups. Subsequently, such extension theory was generalized to Lie algebras by Hochschild \cite{hoch}. See also \cite{ckladra, fial-pen, neeb, fregier, liu-sheng-wang} for recent advances on non-abelian extensions of Lie groups, Lie algebras and Leibniz algebras. Recently, the non-abelian extension theory of Rota-Baxter Lie algebras and Rota-Baxter Leibniz algebras were made in \cite{mishra-das-hazra,guo-hou}.


In the present paper, we define and study non-abelian extensions of averaging Lie algebras. Among others, we define the second non-abelian cohomology group $H^2_\mathrm{nab} (\mathfrak{g}_P, \mathfrak{h}_Q)$ of an averaging Lie algebra $\mathfrak{g}_P$ with values in another averaging Lie algebra $\mathfrak{h}_Q$. In Theorem \ref{isom}, we show that the set of all equivalence classes of non-abelian extensions of $\mathfrak{g}_P$ by $\mathfrak{h}_Q$ is classified by the second non-abelian cohomology group $H^2_\mathrm{nab} (\mathfrak{g}_P, \mathfrak{h}_Q)$.

\subsection{Inducibility of automorphisms and the Wells map.}
The problem of inducibility of a pair of automorphisms of algebraic structures is another trending direction of research. This problem was first considered by Wells \cite{wells} for abstract groups and further studied in \cite{jin-liu,passi}. In the context of Lie algebras, the problem can be stated as follows. Let $0 \rightarrow \mathfrak{h} \xrightarrow{i} \mathfrak{e} \xrightarrow{p}  \mathfrak{g} \rightarrow 0$
be a given (non-)abelian extension of Lie algebras. Then for any $\gamma \in \Aut(\mathfrak{e})$ with $\gamma( \mathfrak{h}) \subset  \mathfrak{h}$, there is a pair of Lie algebra automorphisms $(\gamma|_{\mathfrak{h}}, \overline{\gamma}=p \gamma s) \in \Aut(\mathfrak{h}) \times \Aut(\mathfrak{g})$, where $s$ is a section of the map $p$. This pair of Lie algebra automorphisms is said to be induced by $\gamma$. 
A pair of automorphisms $(\beta, \alpha) \in \Aut(\mathfrak{h}) \times \Aut(\mathfrak{g})$ is called inducible if there exists $\gamma \in \Aut(\mathfrak{e})$ with $\gamma (\mathfrak{h}) \subset \mathfrak{h}$ that induces $(\beta, \alpha)$. The inducibility problem asks to find the obstruction for the inducibility of a pair of Lie algebra automorphisms. When the given extension is abelian, the inducibility problem was addressed in \cite{bar-singh}. More precisely, they defined the Wells map (a generalization of a map considered in \cite{wells} for group extensions) in the context of Lie algebras and showed that the obstruction for the inducibility of a pair of Lie algebra automorphisms can be described by the image of the Wells map.

In this paper, we have undertaken the inducibility problem in the context of averaging Lie algebras. Given a non-abelian extension $0 \rightarrow \mathfrak{h}_Q \xrightarrow{i} \mathfrak{e}_U \xrightarrow{p} \mathfrak{g}_P \rightarrow 0$ of averaging Lie algebras, we first construct an analogue of the Wells map $\mathcal{W} : \mathrm{Aut}(\mathfrak{h}_Q ) \times \mathrm{Aut} (\mathfrak{g}_P) \rightarrow H^2_\mathrm{nab} (\mathfrak{g}_P, \mathfrak{h}_Q)$. Our main result shows that a pair of averaging Lie algebra automorphisms $(\beta, \alpha) \in  \mathrm{Aut}(\mathfrak{h}_Q ) \times \mathrm{Aut} (\mathfrak{g}_P)$ is inducible if and only if $\mathcal{W}((\beta, \alpha)) = 0$ (cf. Theorem \ref{main}). Therefore, the corresponding obstruction (which can be seen as an image of the Wells map) lies in the second non-abelian cohomology group $H^2_\mathrm{nab} (\mathfrak{g}_P, \mathfrak{h}_Q)$. We also construct a short exact sequence, called the Wells exact sequence that connects various automorphism groups and the second non-abelian cohomology group (cf. Theorem \ref{w-e-s}).

Finally, we discuss abelian extensions of an averaging Lie algebra by a given representation and show how the results of non-abelian extensions can be realized in the case of abelian extensions. In particular, we consider split abelian extensions and prove Proposition \ref{last-prop}.

\medskip

\noindent {\bf{Organization of the paper.}} In Section \ref{sec2}, we recall averaging Lie algebras, their representations and (abelian) cohomology. Next, in Section \ref{sec3}, we study homotopy averaging operators on $2$-term $L_\infty$-algebras and characterize some particular classes of $2$-term averaging $L_\infty$-algebras. Non-abelian extensions of an averaging Lie algebra $\mathfrak{g}_P$ by another averaging Lie algebra $\mathfrak{h}_Q$ are considered in Section \ref{sec4}. To classify the equivalence classes of such non-abelian extensions, we define the second non-abelian cohomology group. In Section \ref{sec5}, we primarily focus on the problem of inducibility of a pair of averaging Lie algebra automorphisms in a given non-abelian extension. To approach this problem, we invoke the notion of the Wells map in the context of averaging Lie algebras and show that the corresponding obstruction can be seen as an image of the Wells map. We also construct the Wells exact sequence in the present context. Finally, in Section \ref{sec6}, we discuss abelian extensions of an averaging Lie algebra by a given representation. 

\medskip

\noindent {\bf{Notation.}} All vector spaces, (multi)linear maps, Lie algebras, tensor and wedge products are over a field $\textbf{k}$ of characteristic zero unless specified otherwise.

\section{Averaging Lie algebras}\label{sec2}

In this section, we recall averaging Lie algebras, their representations and cohomology theory. Our main references are \cite{kotov, sheng-embed}.

\begin{defn}\label{aver-on-lie}
Let $\mathfrak{g}= \left(\mathfrak{g}, [~,~]_{\mathfrak{g}}\right)$ be a Lie algebra. An {\bf{averaging operator}} on the Lie algebra $\mathfrak{g}$ is a linear map $P: \mathfrak{g} \rightarrow \mathfrak{g}$ that satisfies 
\begin{equation}{\label{1}}
[P(x), P(y)]_{\mathfrak{g}} = P\left([P(x),y]_{\mathfrak{g}}\right),~\text{for all}~~x,y \in \mathfrak{g}.
\end{equation}
\end{defn}

Note that (\ref{1}) is equivalent to the condition $[P(x),P(y)]_{\mathfrak{g}} = P\left( [x,P(y)]_{\mathfrak{g}}\right),$ {for all} $x,y \in \mathfrak{g}$. Thus, an averaging operator on a Lie algebra can be described by this equivalent condition. 

\begin{exam}
\begin{enumerate}
\item[(i)] The identity map $\Id : \mathfrak{g} \rightarrow \mathfrak{g}$ is an averaging operator on any Lie algebra $\mathfrak{g}$.

\item[(ii)] Let $\mathfrak{g}$ be any Lie algebra. Then the direct sum $\mathfrak{g} \oplus \mathfrak{g}$ carries a Lie algebra structure with the bracket $$\left[(x_1,x_2), (y_1,y_2) \right]_{\ltimes} := \left( [x_1,y_1]_{\mathfrak{g}}, [x_1,y_2]_{\mathfrak{g}} - [y_1,x_2]_{\mathfrak{g}}\right), ~ \text{for}~~(x_1,x_2), (y_1,y_2) \in \mathfrak{g} \oplus \mathfrak{g}.$$ 
With the above Lie bracket on $\mathfrak{g} \oplus \mathfrak{g}$, the map $P: \mathfrak{g} \oplus \mathfrak{g} \rightarrow \mathfrak{g} \oplus \mathfrak{g}$ defined by $P(x_1,x_2) = (x_2,0)$ is an averaging operator on the Lie algebra $\mathfrak{g} \oplus \mathfrak{g}$.

\item[(iii)] For any Lie algebra $\mathfrak{g}$, the direct sum $\underbrace{\mathfrak{g} \oplus \cdots \oplus \mathfrak{g}}_{n \text{~copies}}$ inherits a Lie bracket given by
\begin{align*}
&\{ \! [ (x_1,\dots, x_n), (y_1, \dots, y_n) ] \! \}\\
& \quad := \big( [x_1,y_1]_{\mathfrak{g}}, [x_1,y_2]_{\mathfrak{g}}-[y_1,x_2]_{\mathfrak{g}}, \dots, \underbrace{[x_1,y_i]_{\mathfrak{g}} - [y_1,x_i]_{\mathfrak{g}}}_{i\text{-th place for }i \geq 2}, \dots, [x_1,y_n]_{\mathfrak{g}} - [y_1,x_n]_{\mathfrak{g}}\big),
\end{align*}
for $(x_1, \dots, x_n), (y_1, \dots, y_n) \in \mathfrak{g} \oplus \dots \oplus \mathfrak{g}$. With the above Lie bracket, all the maps $P,Q_2, \dots, Q_n: \mathfrak{g} \oplus \dots \oplus \mathfrak{g} \rightarrow  \mathfrak{g} \oplus \dots \oplus \mathfrak{g}$ given by
\begin{align*}
P(x_1, \dots, x_n)  = (x_2+ \dots + x_n,0, \dots, 0)~~~\text{and} ~~~
Q_i(x_1, \dots, x_n)&= (x_i,0, \dots, 0)~~(i \geq 2)
\end{align*}
are averaging operators on $\mathfrak{g} \oplus \dots \oplus \mathfrak{g}$.

\item[(iv)] Let $A$ be an associative algebra and $P:A \rightarrow A$ be an 
 {averaging operator} on $A$. Then the map $P$ can be regarded as an averaging operator on the associated Lie algebra $(A, [~,~])$. To see this, we observe that 
$$[P(a),P(b)] = P(a)P(b) -P(b)P(a)= P(P(a)b)-P(bP(a))=P([P(a),b]),~\text{for}~~a,b \in A.$$

\item[(v)] Let $\mathfrak{g}$ be a Lie algebra and $(V, \psi)$ be a {representation} of it, i.e., $\psi : \mathfrak{g} \rightarrow \mathrm{End}(V)$ is a Lie algebra homomorphism. An {\em {embedding tensor}} on $\mathfrak{g}$ with respect to the representation $(V, \psi)$ is a linear map $T: V \rightarrow \mathfrak{g}$ that satisfies 
$$[T(u),T(v)]_{\mathfrak{g}} = T(\psi_{T(u)}v), ~\text{for all}~~u,v \in V.$$ 
See \cite{kotov} for more details. Note that given a Lie algebra $\mathfrak{g}$ and a representation $(V, \psi)$, one can construct the semi-direct product Lie algebra on the direct sum $\mathfrak{g} \oplus V$ with the bracket 
$$[(x,u),(y,v)]_\ltimes := ([x,y]_\mathfrak{g}, \psi_xv -\psi_yu),~\text{for}~~(x,u),(y,v) \in \mathfrak{g} \oplus V.$$
Then it is easy to see that a map $T: V \rightarrow \mathfrak{g}$ is an embedding tensor if and only if the map $P_T: \mathfrak{g} \oplus V \rightarrow \mathfrak{g} \oplus V$ defined by $P_T(x,u) = (T(u),0)$ is an averaging operator on the semi-direct product Lie algebra. Thus, averaging operators on Lie algebras are closely related to embedding tensors.
\end{enumerate}
\end{exam} 

\begin{defn}
An {\bf{averaging Lie algebra}} is a Lie algebra $\mathfrak{g}$ equipped with an averaging operator $P: \mathfrak{g} \rightarrow \mathfrak{g}$.
\end{defn}

Throughout this paper, we denote an averaging Lie algebra as above simply by the notation $\mathfrak{g}_P$. Let $\mathfrak{g}_{P}$ and $\mathfrak{g}'_{P'}$ be two averaging Lie algebras. A morphism of averaging Lie algebras from $\mathfrak{g}_P$ to $\mathfrak{g}'_{P'}$ is given by a Lie algebra homomorphism $\tau: \mathfrak{g} \rightarrow \mathfrak{g}{'}$ that satisfies $P{'} \circ \tau = \tau \circ P$. Further, it is said to be an isomorphism if $\tau$ is a linear isomorphism. Given an averaging Lie algebra $\mathfrak{g}_P$, we denote the group of all averaging Lie algebra automorphisms of $\mathfrak{g}_P$ by the notation $\Aut(\mathfrak{g}_P)$.

\begin{remark}\label{X}
Averaging Lie algebras are closely related to Leibniz algebras. Recall that a (left) {{Leibniz algebra}} is a vector space $\ell$ equipped with a bilinear bracket $\left\lbrace ~,~ \right\rbrace; \ell \times \ell \rightarrow \ell$ satisfying 
$$\left\lbrace x, \left\lbrace y,z \right\rbrace \right\rbrace=\left\lbrace \left\lbrace x,y \right\rbrace ,z \right\rbrace + \left\lbrace y, \left\lbrace x,z \right\rbrace \right\rbrace, ~\text{for all}~~x,y,z \in \ell.$$ 
Let $\mathfrak{g}_P$ be an averaging Lie algebra. Then the vector space $\mathfrak{g}$ carries a Leibniz algebra structure (induced from the averaging operator $P$) with the bracket $\left\lbrace x,y \right\rbrace := [P(x),y]_{\mathfrak{g}}$, for $x, y \in \mathfrak{g}$.

Conversely, let $(\ell, \left\lbrace~,~\right\rbrace)$ be a Leibniz algebra. Let $\ell_{\Lie} = \ell / \left\lbrace \ell, \ell \right\rbrace$ be the induced Lie algebra. Then the Lie algebra $\ell_{\Lie}$ has a representation $(\ell, \psi)$ on the vector space $\ell$, where the Lie algebra homomorphism $\psi: \ell_{\Lie} \rightarrow \End(\ell)$ is given by $\psi_{[x]}y := \left\lbrace x,y \right\rbrace$, for $[x] \in \ell_{\Lie}$ and $y \in \ell$. Then the quotient map $q : \ell \rightarrow \ell_{\Lie}$ given by $q(x) = [x]$ is an embedding tensor. Therefore, the map $$P: \ell_{\Lie} \oplus \ell \rightarrow \ell_{\Lie} \oplus \ell~~\text{given by}~~P([x],y)=([y],0)$$ is an averaging operator on the semi-direct product Lie algebra $\ell_{\Lie} \oplus \ell$.
\end{remark}

\begin{defn}\label{rep}
Let $\mathfrak{g}_{P}$ be an averaging Lie algebra. A {\bf{representation}} of $\mathfrak{g}_P$ consists of a Lie algebra representation $(V,\psi)$ with a linear map $Q: V \rightarrow V$ satisfying 
$$\psi_{P(x)}Q(v)= Q (\psi_{P(x)}v) =  Q (\psi_x Q(v)),~\text{for all}~~ x \in \mathfrak{g}, v \in V.$$
\end{defn}

We denote a representation as above simply by $V_Q$ when the action map $\psi$ is clear from the context. Note that any averaging Lie algebra $\mathfrak{g}_P$ can be regarded as a representation of itself, where $\mathfrak{g}$ is equipped with the adjoint Lie algebra representation. 

Next, we recall the cohomology of an averaging Lie algebra with coefficients in a representation \cite{sheng-embed}. Let $\mathfrak{g}_P$ be an averaging Lie algebra and $V_Q$ be a representation of it. For each $n \geq 0$, the $n$-th cochain group $C^{n}_{\ALie}(\mathfrak{g}_P,V_Q)$ is given by
\begin{align*}
C^{n}_{\ALie}(\mathfrak{g}_P,V_Q) = 
\begin{cases}
0 &\text{if}~~n=0,\\
\Hom(\mathfrak{g},V) &\text{if}~~n=1,\\
\Hom(\wedge^{n} \mathfrak{g}, V) \oplus \Hom(\mathfrak{g}^{\otimes n-1}, V) &\text{if}~~n \geq 2.
\end{cases}
\end{align*}
There is a map $\delta^{n}_{\ALie}: C^{n}_{\ALie}(\mathfrak{g}_P, V_Q) \rightarrow C^{n+1}_{\ALie}(\mathfrak{g}_P, V_Q)$ given by 
$$\delta^{n}_{\ALie}((f, \theta))=(\delta^{n}_{\Lie}(f), \partial^{n-1}_{\Leib}(\theta) + (-1)^{n}f \circ P^{\otimes n} - (-1)^{n}~Qf \circ (P^{\otimes n-1} \otimes \Id)),~\text{for}~~(f,\theta) \in C^{n}_{\ALie}(\mathfrak{g}_P, V_Q).$$ 
Here $\delta^{n}_{\Lie}: \Hom(\wedge^{n} \mathfrak{g} , V) \rightarrow \Hom(\wedge^{n+1}\mathfrak{g}, V)$ is the standard Chevalley-Eilenberg coboundary operator of the Lie algebra $\mathfrak{g}$ with coefficients in the representation $(V,\psi)$, and  $\partial^{n-1}_{\Leib}: \Hom(\mathfrak{g}^{\otimes n-1} , V) \rightarrow \Hom(\mathfrak{g}^{\otimes n}, V)$ is the Loday-Pirashvili coboundary operator of the induced Leibniz algebra $(\mathfrak{g}, \left\lbrace~,~\right\rbrace)$ with coefficients in a suitable representation on $V$ \cite{sheng-embed}. Explicitly, the map $\partial^{n-1}_{\Leib}$ is given by
\begin{align*}
\partial^{n-1}_{\Leib}(\theta)(x_1, \dots, x_n) :=& \sum_{i=1}^{n-1}(-1)^{i+1}~\psi_{P(x_i)}\theta(x_1, \dots, \widehat{x_i}, \ldots, x_n)\\
&+  (-1)^{n+1}~ \psi_{P(x_n)}\theta(x_1, \dots, x_{n-1})+ (-1)^{n}~Q(\psi_{x_n}\theta(x_1, \dots, x_{n-1})) \\
                                                &+  \sum_{1 \leq i < j \leq n}(-1)^{i}~\theta(x_1, \dots, \widehat{x_i}, \dots, x_{j-1},[P(x_i),x_j]_{\mathfrak{g}}, \dots , x_n),
                                                \end{align*}
for $\theta \in \Hom(\mathfrak{g}^{\otimes n-1}, V)~~\text{and}~~x_1, \dots, x_n \in \mathfrak{g}$.
Then it turns out that $\left\lbrace C^{\bullet}_{\ALie}(\mathfrak{g}_P, V_Q), \delta^{\bullet}_{\ALie}\right\rbrace$ is a cochain complex. The corresponding cohomology is called the cohomology of the averaging Lie algebra $\mathfrak{g}_P$ with coefficients in the representation $V_Q$. We denote the $n$-th cohomology group by $H^{n}_{\ALie}(\mathfrak{g}_P, V_Q)$.

\begin{remark}
In \cite{sheng-embed} the authors also considered central extensions of an averaging Lie algebra and showed that the set of all equivalence classes of central extensions is classified by the second cohomology group with trivial representation.
\end{remark}

\section{2-term homotopy averaging Lie algebras}\label{sec3}
The concept of $L_\infty$-algebras is introduced in \cite{stas,lada-markl} as the homotopy version of Lie algebras. In this section, we first introduce homotopy averaging operators on $2$-term $L_\infty$-algebras. A $2$-term $L_\infty$-algebra equipped with a homotopy averaging operator is called a $2$-term averaging $L_\infty$-algebra. We focus on `skeletal' and `strict' $2$-term averaging $L_\infty$-algebras. In particular, we show that skeletal $2$-term averaging $L_\infty$-algebras correspond to third cocycles of averaging Lie algebras. Next, we introduce crossed modules of averaging Lie algebras and show that crossed modules of averaging Lie algebras correspond to strict $2$-term averaging $L_\infty$-algebras.
\begin{defn}\cite{baez}
A {\bf $2$-term $L_\infty$-algebra} consists of a triple $\mathcal{G} = (\mathfrak{g}_1 \xrightarrow{d} \mathfrak{g}_0, \llbracket ~, ~ \rrbracket, l_3)$ in which $\mathfrak{g}_1 \xrightarrow{d} \mathfrak{g}_0$ is a $2$-term chain complex, $\llbracket ~, ~ \rrbracket : \mathfrak{g}_i \times \mathfrak{g}_j \rightarrow \mathfrak{g}_{i+j}$ (for $0 \leq i, j \leq 1$) is a bilinear map and $l_3 : \mathfrak{g}_0 \times \mathfrak{g}_0 \times \mathfrak{g}_0 \rightarrow \mathfrak{g}_1$ is a skew-symmetric trilinear map satisfying for all $w, x, y, z \in \mathfrak{g}_0$ and $h, k \in \mathfrak{g}_1$,
\begin{itemize}
    \item[(L1)] $\llbracket x, y \rrbracket = - \llbracket y, x \rrbracket$,
    \item[(L2)] $\llbracket x, h \rrbracket = - \llbracket h, x \rrbracket$,
    \item[(L3)] $\llbracket h, k \rrbracket = 0$,
    \item[(L4)] $d \llbracket x, h \rrbracket = \llbracket x, dh \rrbracket$,
    \item[(L5)] $\llbracket dh, k \rrbracket= \llbracket h, dk \rrbracket$,
    \item[(L6)] $d \big(  l_3 (x, y, z)  \big) =   \llbracket x, \llbracket y, z \rrbracket \rrbracket +  \llbracket y, \llbracket z, x \rrbracket \rrbracket +  \llbracket z, \llbracket x, y \rrbracket \rrbracket$,
    \item[(L7)] $l_3 (x, y, dh) =  \llbracket x, \llbracket y, h \rrbracket \rrbracket +  \llbracket y, \llbracket h, x \rrbracket \rrbracket +  \llbracket h, \llbracket x, y \rrbracket \rrbracket$, 
    \item[(L8)] $\llbracket w, l_3 (x, y, z) \rrbracket      -     \llbracket  x, l_3 (w, y, z) \rrbracket    +     \llbracket  y, l_3 (w, x, z) \rrbracket    -   \llbracket z, l_3 (w, x, y) \rrbracket \\ 
      = l_3   (\llbracket w, x \rrbracket, y, z ) - l_3   (\llbracket  w, y \rrbracket, x, z ) + l_3   (\llbracket w, z \rrbracket, x, y ) + l_3   (\llbracket x, y \rrbracket, w, z ) - l_3   (\llbracket x, z \rrbracket , w, y) + l_3   (\llbracket y, z  \rrbracket, w, x).$
\end{itemize}
\end{defn}

We now introduce the notion of homotopy averaging operator on a $2$-term $L_\infty$-algebra. This can be realized as the homotopy version of averaging operators on Lie algebras.

\begin{defn}
    Let $\mathcal{G} = (\mathfrak{g}_1 \xrightarrow{d} \mathfrak{g}_0, \llbracket ~, ~ \rrbracket, l_3)$ be a $2$-term $L_\infty$-algebra. A {\bf homotopy averaging operator} on $\mathcal{G}$ is a triple $\mathcal{P} = (P_0, P_1, P_2)$ consists of linear maps $P_0 : \mathfrak{g}_0 \rightarrow \mathfrak{g}_0$ and $P_1 : \mathfrak{g}_1 \rightarrow \mathfrak{g}_1$, and a skew-symmetric bilinear map $P_2 : \mathfrak{g}_0 \times \mathfrak{g}_0 \rightarrow \mathfrak{g}_1$ such that for all $x, y, z \in \mathfrak{g}_0$ and $h \in \mathfrak{g}_1$,
    \begin{itemize}
    \item[(A1)] $P_0 \circ d = d \circ P_1$,
        \item[(A2)] $d ( P_2 (x, y)) = P_0 ( \llbracket P_0 (x), y \rrbracket) - \llbracket P_0 (x) , P_0 (y) \rrbracket,$
        \item[(A3)] $P_2 (x, dh) = P_1 ( \llbracket P_0 (x), h \rrbracket ) - \llbracket P_0 (x), P_1 (h) \rrbracket = P_1 ( \llbracket x, P_1 (h) \rrbracket ) - \llbracket P_0 (x) , P_1 (h) \rrbracket,$
        \item[(A4)] $\llbracket P_0 (x) , P_2 (y, z) \rrbracket - \llbracket P_0 (y) , P_2 (x, z) \rrbracket  + \llbracket P_0 (z) , P_2 (x, y) \rrbracket  - P_1 \llbracket z, P_2 (x, y)  \rrbracket  - P_2 (\llbracket  P_0 (x), y \rrbracket, z  )\\
        - P_2 (y, \llbracket P_0 (x), z \rrbracket) + P_2 (x, \llbracket P_0 (y), z \rrbracket) = l_3 (P_0 (x), P_0 (y), P_0 (z)) - P_1 l_3 (P_0 (x), P_0 (y), z)$.
    \end{itemize}
\end{defn}

A {\bf $2$-term averaging $L_\infty$-algebra} is a $2$-term $L_\infty$-algebra $\mathcal{G} = (\mathfrak{g}_1 \xrightarrow{d} \mathfrak{g}_0, \llbracket ~, ~ \rrbracket, l_3)$ equipped with a homotopy averaging operator $\mathcal{P} = (P_0, P_1, P_2)$ on it. We denote a $2$-term averaging $L_\infty$-algebra as above by $( \mathfrak{g}_1 \xrightarrow{d} \mathfrak{g}_0, \llbracket ~, ~ \rrbracket, l_3, P_0, P_1, P_2   )$ or simply by $\mathcal{G}_\mathcal{P}$.

\begin{defn}
    Let $\mathcal{G}_\mathcal{P} = ( \mathfrak{g}_1 \xrightarrow{d} \mathfrak{g}_0, \llbracket ~, ~ \rrbracket, l_3, P_0, P_1, P_2   )$ be a $2$-term averaging $L_\infty$-algebra. It is said to be
    \begin{itemize}
        \item[(i)] {\bf skeletal} if $d = 0$,
        \item[(ii)] {\bf strict} if $l_3 = 0$ and $P_2 = 0.$
    \end{itemize}
\end{defn}

The following result gives a characterization of skeletal $2$-term averaging $L_\infty$-algebras in terms of $3$-cocycles of averaging Lie algebras.

\begin{prop}\label{skeletal-1}
    There is a $1-1$ correspondence between skeletal $2$-term averaging $L_\infty$-algebras and triples of the form $\big(\mathfrak{g}_P, V_Q, (f, \theta) \big)$, where $\mathfrak{g}_P$ is an averaging Lie algebra, $V_Q$ is a representation and $(f, \theta) \in C^3_\mathrm{ALie} (\mathfrak{g}_P, V_Q)$ is a $3$-cocycle.
\end{prop}

\begin{proof}
Let $\mathcal{G}_\mathcal{P} = ( \mathfrak{g}_1 \xrightarrow{0} \mathfrak{g}_0, \llbracket ~, ~ \rrbracket, l_3, P_0, P_1, P_2   )$ be a skeletal $2$-term averaging $L_\infty$-algebra. Then it follows from conditions (L1), (L6) and (A2) that $(\mathfrak{g}_0, \llbracket ~, ~\rrbracket)$ is a Lie algebra and $P_0 : \mathfrak{g}_0 \rightarrow \mathfrak{g}_0$ is an averaging operator on it. Thus, $(\mathfrak{g}_0)_{P_0}$ is an averaging Lie algebra. On the other hand, by conditions (L7) and (A3), we get that $(\mathfrak{g}_1)_{P_1}$ is a representation of the averaging Lie algebra $(\mathfrak{g}_0)_{P_0}$ with the Lie algebra action map $\psi : \mathfrak{g}_0 \rightarrow \mathrm{End}(\mathfrak{g}_1)$, $\psi_x h = \llbracket x, h \rrbracket$, for $x \in \mathfrak{g}_0$, $h \in \mathfrak{g}_1$. With the above averaging Lie algebra and its representation, the conditions (L8) and (A4) are respectively equivalent to
\begin{align*}
    \big(  \delta^3_\mathrm{Lie} (l_3)  \big)(w, x, y, z)=0 ~~~~ \text{ and } ~~~~ \big(  \partial^2_\mathrm{Leib} (P_2)  \big) (x, y, z) - l_3 \big(  P_0 (x), P_0 (y), P_0 (z)  \big) + P_1 l_3 \big( P_0 (x), P_0 (y), z  \big) = 0.
\end{align*}
Therefore, $\delta^3_\mathrm{ALie} ((l_3, P_2) = \big( \delta^3_\mathrm{Lie} (l_3)  , ~ \partial^2_\mathrm{Leib} (P_2) - l_3 \circ P_0^{\otimes 3} - P_1 l_3 (P_0^{\otimes 2} \otimes \mathrm{Id})  \big) = 0$. Hence we obtain a required triple $\big(  (\mathfrak{g}_0)_{P_0}, (\mathfrak{g}_1)_{P_1}, (l_3, P_2)   \big).$

Conversely, let $(\mathfrak{g}_P, V_Q, (f, \theta))$ be a triple in which $\mathfrak{g}_P$ is an averaging Lie algebra, $V_Q$ is a representation (with the action map $\psi : \mathfrak{g} \rightarrow \mathrm{End}(V)$) and $(f, \theta) \in C^3_\mathrm{ALie} (\mathfrak{g}_P, V_Q)$ is a $3$-cocycle. Then it is easy to verify that $(V \xrightarrow{0} \mathfrak{g}, \llbracket ~, ~ \rrbracket, f, P, Q, \theta)$ is a skeletal $2$-term averaging $L_\infty$-algebra, where the bilinear map $\llbracket ~, ~ \rrbracket$ is given by
\begin{align*}
    \llbracket x, y \rrbracket := [x, y]_\mathfrak{g}, ~~~~ \llbracket x, v \rrbracket = - \llbracket v, x \rrbracket := \psi_x v ~~~~ \text{ and } ~~~~ \llbracket u, v \rrbracket := 0,  \text{ for } x, y \in \mathfrak{g}, ~ u, v \in V.
\end{align*}
This completes the proof.
\end{proof}

The above result motivates us to consider the following notion. Let $\mathcal{G}_\mathcal{P} = ( \mathfrak{g}_1 \xrightarrow{0} \mathfrak{g}_0, \llbracket ~, ~ \rrbracket, l_3, P_0, P_1, P_2   )$ and $\mathcal{G}'_{\mathcal{P}'} = ( \mathfrak{g}_1 \xrightarrow{0} \mathfrak{g}_0, \llbracket ~, ~ \rrbracket', l'_3, P'_0, P'_1, P'_2   )$ be two skeletal $2$-term averaging $L_\infty$-algebras  on the same chain complex. They are said to be equivalent if
\begin{align*}
    \llbracket ~, ~ \rrbracket = \llbracket ~, ~ \rrbracket', \quad  P_0 = P_0', \quad P_1 = P_1'
\end{align*}
and there exist a skew-symmetric bilinear map $ g: \mathfrak{g}_0 \times \mathfrak{g}_0 \rightarrow \mathfrak{g}_1$ and a linear map $\vartheta : \mathfrak{g}_0 \rightarrow \mathfrak{g}_1$ such that
\begin{align*}
    (l'_3, P_2') = (l_3, P_2) + \delta^2_\mathrm{ALie} ((g, \vartheta)), 
\end{align*}
where $\delta^2_\mathrm{ALie}$ is the coboundary operator of the averaging Lie algebra $(\mathfrak{g}_0)_{P_0}$ with coefficients in the representation $(\mathfrak{g}_1)_{P_1}$.

With the above notion of equivalence, Proposition \ref{skeletal-1} can be strengthened into the following result.

\begin{thm}\label{skeletal-thm}
    There is a $1-1$ correspondence between equivalence classes of skeletal $2$-term averaging $L_\infty$-algebras and triples of the form $\big(\mathfrak{g}_P, V_Q, [(f, \theta)] \big)$, where $\mathfrak{g}_P$ is an averaging Lie algebra, $V_Q$ is a representation and $[(f, \theta)] \in H^3_\mathrm{ALie} (\mathfrak{g}_P, V_Q)$ is a third cohomology class.
\end{thm}

Next, we introduce crossed modules of averaging Lie algebras and characterize strict $2$-term averaging $L_\infty$-algebras.

\begin{defn}
    A {\bf crossed module} of averaging Lie algebras is a quadruple $\big(  (\mathfrak{g}_1)_{P_1}, (\mathfrak{g}_0)_{P_0} , d, \rho \big)$, where $(\mathfrak{g}_1)_{P_1}$ and $(\mathfrak{g}_0)_{P_0}$ are both averaging Lie algebras, $d : (\mathfrak{g}_1)_{P_1} \rightarrow (\mathfrak{g}_0)_{P_0}$ is an averaging Lie algebra morphism and $\rho : \mathfrak{g}_0 \rightarrow \mathrm{Der} (\mathfrak{g}_1)$ is a Lie algebra homomorphism that makes $(\mathfrak{g}_1)_{P_1}$ into a representation of the averaging Lie algebra $(\mathfrak{g}_0)_{P_0}$ satisfying additionally
    \begin{align*}
        d (\rho_x h) = [x, dh]_{\mathfrak{g}_0} ~~~~ \text{ and } ~~~~ \rho_{dh} k  = [h, k]_{\mathfrak{g}_1}, \text{ for all } x \in \mathfrak{g}_0 \text{ and } h, k \in \mathfrak{g}_1.
    \end{align*}
\end{defn}

\begin{prop}\label{crossed-semi}
    Let $\big(  (\mathfrak{g}_1)_{P_1}, (\mathfrak{g}_0)_{P_0} , d, \rho \big)$ be a crossed module of averaging Lie algebras. Then $(\mathfrak{g}_0 \oplus \mathfrak{g}_1)_{P_0 \oplus P_1}$ is an averaging Lie algebra, where $\mathfrak{g}_0 \oplus \mathfrak{g}_1$ is equipped with the bracket
    \begin{align}\label{crossed-semidi}
        [(x, h), (y, k)] := ([x, y]_{\mathfrak{g}_0}, \rho_x k - \rho_y k + [h, k]_{\mathfrak{g}_1}), \text{ for } (x, h), (y, k) \in \mathfrak{g}_0 \oplus \mathfrak{g}_1.
    \end{align}
\end{prop}

\begin{proof}
    Since $\mathfrak{g}_0$, $\mathfrak{g}_1$ are both Lie algebras and $\rho : \mathfrak{g}_0 \rightarrow \mathrm{Der}(\mathfrak{g}_1)$ is a Lie algebra homomorphism, it follows that $\mathfrak{g}_0 \oplus \mathfrak{g}_1$ is a Lie algebra with the bracket (\ref{crossed-semidi}). Moreover, for any $(x, h), (y, k) \in \mathfrak{g}_0 \oplus \mathfrak{g}_1$, we have
    \begin{align*}
        &[  (P_0 \oplus P_1) (x, h) , (P_0 \oplus P_1) (y, k)] \\
        &= [  (P_0 (x), P_1 (h)), (P_0 (y), P_1 (k)) ] \\
        &= \big(   [P_0 (x), P_0 (y)]_{\mathfrak{g}_0}, ~ \rho_{P_0 (x)} P_1 (k) - \rho_{P_0 (y)} P_1 (h) + [P_1 (h), P_1 (k)]_{\mathfrak{g}_1}  \big) \\
        &= \big(   P_0 [P_0 (x), y]_{\mathfrak{g}_0}, ~ P_1 (\rho_{P_0 (x)} k) - P_1 (\rho_y P_1 (h)) + P_1 [P_1 (h), k]_{\mathfrak{g}_1} \big) \\
        &= (P_0 \oplus P_1) [  (P_0 (x), P_1 (h)), (y, k)] \\
        &= (P_0 \oplus P_1) [   (P_0 \oplus P_1) (x, h), (y, k)].
    \end{align*}
    This shows that the map $P_0 \oplus P_1 : \mathfrak{g}_0 \oplus \mathfrak{g}_1 \rightarrow \mathfrak{g}_0 \oplus \mathfrak{g}_1$ is an averaging operator. This proves the result.
\end{proof}

\begin{thm}\label{crossed-11}
    There is a $1-1$ correspondence between strict $2$-term averaging $L_\infty$-algebras and crossed modules of averaging Lie algebras.
\end{thm}

\begin{proof}
    Let $\mathcal{G}_\mathcal{P} = (  \mathfrak{g}_1 \xrightarrow{d} \mathfrak{g}_0, \llbracket ~, ~ \rrbracket, l_3 = 0, P_0 , P_1 , P_2 = 0)$ be a strict $2$-term averaging $L_\infty$-algebra. Then it follows from (L1), (L6) and (A2) that $(\mathfrak{g}_0, \llbracket ~, ~ \rrbracket)$ is a Lie algebra and $P_0 : \mathfrak{g}_0 \rightarrow \mathfrak{g}_0$ is an averaging operator on it, i.e., $(\mathfrak{g}_0)_{P_0}$ is an averaging Lie algebra. Next, we define a bilinear bracket $[~,~]_{\mathfrak{g}_1} : \mathfrak{g}_1 \times \mathfrak{g}_1 \rightarrow \mathfrak{g}_1$ by $[h, k ]_{\mathfrak{g}_1} := \llbracket dh, k \rrbracket$, for $h, k \in \mathfrak{g}_1$. Then from conditions (L2), (L5) and (L7) that $(\mathfrak{g}_1, [~,~]_{\mathfrak{g}_1})$ is a Lie algbera. Moreover, condition (A3) yields that $P_1 : \mathfrak{g}_1 \rightarrow \mathfrak{g}_1$ is an averaging operator. Hence $(\mathfrak{g}_1)_{P_1}$ is also an averaging Lie algebra. On the other hand, the conditions (L4) and (A1) implies that $d: (\mathfrak{g}_1)_{P_1} \rightarrow (\mathfrak{g}_0)_{P_0}$ is an averaging Lie algebra morphism. Finally, we define a map $\rho : \mathfrak{g}_0 \rightarrow \mathrm{Der} (\mathfrak{g}_1)$ by  $\rho_x h := \llbracket x, h \rrbracket$, for $x \in \mathfrak{g}_0$, $h \in \mathfrak{g}_1$. Then it follows from (L7) and (A3) that $\rho$ makes $(\mathfrak{g}_1)_{P_1}$ into a representation of the averaging Lie algebra $(\mathfrak{g}_0)_{P_0}$. We also have
    \begin{align*}
         d (\rho_x h) = d \llbracket x, h \rrbracket = \llbracket x, dh \rrbracket ~~~ \text{ and } ~~~~ \rho_{dh} k = \llbracket dh, k \rrbracket = [ h,k ]_{ \mathfrak{g}_1 }, \text{ for } x \in \mathfrak{g}_0, h, k \in \mathfrak{g}_1.
    \end{align*}
    Hence $\big(    (\mathfrak{g}_1)_{P_1},  (\mathfrak{g}_0)_{P_0}, d, \rho \big)$ is a crossed module of averaging Lie algebras.

    Conversely, let $\big(    (\mathfrak{g}_1)_{P_1},  (\mathfrak{g}_0)_{P_0}, d, \rho \big)$ be a crossed module of averaging Lie algebras. Then it is easy to verify that $( \mathfrak{g}_1 \xrightarrow{d} \mathfrak{g}_0, \llbracket ~, ~ \rrbracket, l_3 = 0, P_0 , P_1 , P_2 = 0)$ is a strict $2$-term averaging $L_\infty$-algebra, where the bracket $\llbracket ~, ~ \rrbracket : \mathfrak{g}_i \times \mathfrak{g}_j \rightarrow \mathfrak{g}_{i+j}$ (for $0 \leq i, j \leq 1$) is given by 
    \begin{align*}
        \llbracket x, y \rrbracket : = [x, y]_{\mathfrak{g}_0}, ~~~~ \llbracket x, h \rrbracket= - \llbracket h, x \rrbracket := \rho_x h ~~~ \text{ and } ~~~~ \llbracket h, k \rrbracket := 0, \text{ for } x, y \in \mathfrak{g}_0, h, k \in \mathfrak{g}_1.
    \end{align*}
    The above two correspondences are inverse to each other. This completes the proof.
\end{proof}

Combining Proposition \ref{crossed-semi} and Theorem \ref{crossed-11}, we obtain the following result.

\begin{prop}
    Let $\mathcal{G}_\mathcal{P} = (\mathfrak{g}_1 \xrightarrow{d} \mathfrak{g}_0, \llbracket ~, ~ \rrbracket, l_3 = 0, P_0, P_1, P_2 = 0)$ be a strict $2$-term averaging $L_\infty$-algebra. Then $(\mathfrak{g}_0 \oplus \mathfrak{g}_1)_{P_0 \oplus P_1}$ is an averaging Lie algebra with the Lie bracket
    \begin{align*}
        [(x, h), (y, k)] := (\llbracket x, y \rrbracket, \llbracket x, k \rrbracket - \llbracket y, h \rrbracket + \llbracket dh , k \rrbracket), \text{ for } (x, h), (y, k) \in \mathfrak{g}_0 \oplus \mathfrak{g}_1.
    \end{align*}
\end{prop}

\medskip

\begin{exam}
    Let $\mathfrak{g}_P$ be an averaging Lie algebra. Then $(\mathfrak{g}_P, \mathfrak{g}_P, \mathrm{Id}, ad)$ is a crossed module of averaging Lie algebras, where $ad$ denotes the adjoint representation. Therefore, it follows that 
    \begin{align*}
        (\mathfrak{g} \xrightarrow{\mathrm{Id}} \mathfrak{g}, [~,~]_\mathfrak{g}, l_3 = 0, P_0 = P, P_1 = P, P_2 =0)
    \end{align*}
    is a strict $2$-term averaging $L_\infty$-algebra.
\end{exam}

More generally, let $\mathfrak{g}_P$ be an averaging Lie algebra and $\mathfrak{h} \subset \mathfrak{g}$ be a Lie ideal that satisfies $P(\mathfrak{h}) \subset \mathfrak{h}$. Then $(\mathfrak{h}_P, \mathfrak{g}_P, i, ad)$ is a crossed module of averaging Lie algebras, where $i: \mathfrak{h} \rightarrow \mathfrak{g}$ is the inclusion map. Hence $(\mathfrak{h} \xrightarrow{i} \mathfrak{g}, [~,~]_\mathfrak{g}, l_3 = 0, P_0 = P, P_1 = P, P_2 = 0)$ is a strict $2$-term averaging $L_\infty$-algebra. As a particular case, we also get the following.

\begin{exam}
    Let $\mathfrak{g}_P$, $\mathfrak{h}_Q$ be two averaging Lie algebras and $f: \mathfrak{g}_P \rightarrow \mathfrak{h}_Q$ be an averaging Lie algebra morphism. Then $(\mathrm{Ker }f, \mathfrak{g}, i, ad)$ is a crossed module of averaging Lie algebras.
\end{exam}

\section{Non-abelian extensions of averaging Lie algebras}\label{sec4}

In this section, we consider non-abelian extensions of an averaging Lie algebra $\mathfrak{g}_P$ by another averaging Lie algebra $\mathfrak{h}_Q$. We also introduce the second non-abelian cohomology space $H^{2}_{\nab}(\mathfrak{g}_P, \mathfrak{h}_Q)$ that classifies the set of all equivalence classes of non-abelian extension of $\mathfrak{g}_P$ by $\mathfrak{h}_Q$.

\begin{defn}\label{defn-nabe}
\begin{enumerate}
\item[(i)] Let $\mathfrak{g}_P$ and $\mathfrak{h}_Q$ be two averaging Lie algebras. A {\bf non-abelian extension} of $\mathfrak{g}_P$ by $\mathfrak{h}_Q$ is an averaging Lie algebra $\mathfrak{e}_U$ together with a short exact sequence 
\begin{equation}\label{nab}
\xymatrix{0 \ar[r] & \mathfrak{h}_Q \ar[r]^{i} & \mathfrak{e}_U \ar[r]^{p}  &  \mathfrak{g}_P \ar[r] & 0}
\end{equation}
of averaging Lie algebras. Often we denote a non-abelian extension as above simply by $\mathfrak{e}_U$ when the underlying short exact sequence is clear from the context.

\item[(ii)] Let $\mathfrak{e}_U$ and $\mathfrak{e{'}}_{U{'}}$ be two non-abelian extensions of $\mathfrak{g}_P$ by $\mathfrak{h}_Q$. Then they are said to be {\bf{equivalent}} if there exists a morphism $\tau:  \mathfrak{e}_U \rightarrow \mathfrak{e{'}}_{U{'}}$ of averaging Lie algebras that makes the following diagram commutative
\begin{align}\label{abl-fig}
\xymatrix{
0 \ar[r] & \mathfrak{h}_Q \ar[r]^{i} \ar@{=}[d] & \mathfrak{e}_U \ar[r]^{p} \ar[d]^{\tau} &  \mathfrak{g}_P \ar[r] \ar@{=}[d] & 0 \\
0 \ar[r] & \mathfrak{h}_Q  \ar[r]_{i{'}} & \mathfrak{e}{'}_{U{'}}  \ar[r]_{p{'}} & \mathfrak{g}_P \ar[r] & 0.
}
\end{align}
The set of all equivalence classes of non-abelian extensions of $\mathfrak{g}_P$ by $\mathfrak{h}_Q$ is denoted by $\Ext_{\nab}(\mathfrak{g}_P,\mathfrak{h}_Q)$.
\end{enumerate}
\end{defn}

Let $\big(  (\mathfrak{g}_1)_{P_1}, (\mathfrak{g}_0)_{P_0}, d , \rho \big)$ be a crossed module of averaging Lie algebras. Then the exact sequence
\begin{align*}
    0 \rightarrow  (\mathfrak{g}_1)_{P_1} \xrightarrow{i}  (\mathfrak{g}_0 \oplus \mathfrak{g}_1)_{P_0 \oplus P_1} \xrightarrow{p}  (\mathfrak{g}_0)_{P_0} \rightarrow 0
\end{align*}
is a non-abelian extension of $(\mathfrak{g}_0)$ by $(\mathfrak{g}_1)_{P_1}$, where the averaging Lie algebra structure on $(\mathfrak{g}_0 \oplus \mathfrak{g}_1)_{P_0 \oplus P_1}$ is given in Proposition \ref{crossed-semi}. Thus, it follows from Theorem \ref{crossed-11} that a strict $2$-term averaging $L_\infty$-algebra gives rise to a non-abelian extension of averaging Lie algebras. 

In Remark \ref{X}, we have seen that an averaging Lie algebra induces a Leibniz algebra structure. This construction is also functorial. Hence a non-abelian extension of an averaging Lie algebra naturally gives rise to a non-abelian extension of (induced) Leibniz algebra in the sense of \cite{liu-sheng-wang}. Similarly, an equivalence between two non-abelian extensions of an averaging Lie algebra naturally gives rise to an equivalence between the corresponding non-abelian extensions of the (induced) Leibniz algebra. 

Let (\ref{nab}) be a non-abelian extension of an averaging Lie algebra $\mathfrak{g}_P$ by another averaging Lie algebra $\mathfrak{h}_Q$. A section of (\ref{nab}) is a linear map $s: \mathfrak{g} \rightarrow \mathfrak{e}$ that satisfies $p \circ s = \Id_{\mathfrak{g}}$. Note that section always exists. Let $s: \mathfrak{g} \rightarrow \mathfrak{e}$ be any section. We define maps $\chi: \wedge^{2}\mathfrak{g} \rightarrow \mathfrak{h}$, $\psi: \mathfrak{g} \rightarrow \Der(\mathfrak{h})$ and $\Phi: \mathfrak{g} \rightarrow \mathfrak{h}$ by 
\begin{align}
\chi(x,y)   :=~& \left[s(x), s(y) \right]_{\mathfrak{e}} - s\left[x,y \right]_{\mathfrak{g}}, \label{chi} \\
\psi_{x}(h) :=~& \left[s(x), h \right]_{\mathfrak{e}}, \label{Psi}\\
\Phi(x)     :=~& U(s(x)) - s(P(x)), \label{Phi}
\end{align}
for $x,y \in \mathfrak{g}$ and $h \in \mathfrak{h}.$
Since (\ref{nab}) defines a non-abelian extension of the Lie algebra $\mathfrak{g}$ by the Lie algebra $\mathfrak{h}$ (by forgetting the average operators), it follows from \cite{fregier} that 
\begin{equation}\label{A}
\psi_{x}\psi_{y}(h) - \psi_{y}\psi_{x}(h) -\psi_{[x,y]_{\mathfrak{g}}}(h) = \left[ \chi(x,y),h \right]_{\mathfrak{h}},
\end{equation}
\begin{equation}\label{B}
\psi_{x}\chi(y,z) + \psi_{y}\chi(z,x)+ \psi_{z}\chi(x,y) - \chi([x,y]_{\mathfrak{g}},z) - \chi([y,z]_{\mathfrak{g}},x)- \chi([z,x]_{\mathfrak{g}},y)=0,
\end{equation}
for all $x,y,z \in \mathfrak{g}$ and $h \in \mathfrak{h}$. Moreover, we have the following. 
\begin{lemma}
The maps $\chi, \psi$ and $\Phi$ defined above satisfy the following compatible conditions: for all $x, y \in \mathfrak{g}$ and $h \in \mathfrak{h}$, 
\begin{equation}\label{C}
\psi_{P(x)}Q(h)= Q(\psi_{P(x)}h) + Q[\Phi(x),h]_{\mathfrak{h}} - [\Phi(x),Q(h)]_{\mathfrak{h}}= Q(\psi_{x}Q(h)) - [\Phi(x),Q(h)]_{\mathfrak{h}},
\end{equation}
\begin{equation}\label{D}
\chi(P(x),P(y)) -Q(\chi(P(x),y))- \Phi[P(x),y]_{\mathfrak{g}} + \psi_{P(x)}\Phi(y) - \psi_{P(y)}\Phi(x) +Q(\psi_y\Phi(x)) + [\Phi(x),\Phi(y)]_{\mathfrak{h}}=0.
\end{equation}
\end{lemma}

\begin{proof}
For any $x \in \mathfrak{g}$ and $h \in \mathfrak{h}$, we have 
\begin{align*}
&~~\psi_{P(x)}Q(h) - Q(\psi_{P(x)}h) - Q[\Phi(x),h]_{\mathfrak{h}} + [\Phi(x), Q(h)]_{\mathfrak{h}}\\
&= [sP(x),Q(h)]_{\mathfrak{e}} -Q[sP(x),h]_{\mathfrak{e}} -Q[Us(x),h]_{\mathfrak{e}} +Q[sP(x),h]_{\mathfrak{e}} + [Us(x),Q(h)]_{\mathfrak{e}} - [sP(x),Q(h)]_{\mathfrak{e}}\\
&= -Q[Us(x), h]_{\mathfrak{e}} + [Us(x), U(h)]_{\mathfrak{e}} \quad (\text{as}~~Q = U \mid_{\mathfrak{h}})\\
&= 0 \quad (\text{as}~~U~~\text{is an embedding tensor}).
\end{align*}
We also have 
\begin{align*}
&~~\psi_{P(x)}Q(h) - Q(\psi_{x}Q(h)) + [\Phi(x),Q(h)]_{\mathfrak{h}}\\
&= [sP(x),Q(h)]_{\mathfrak{e}} - Q[s(x),Q(h)]_{\mathfrak{e}} + [Us(x),Q(h)]_{\mathfrak{e}} - [sP(x),Q(h)]_{\mathfrak{e}}\\
&=  - Q[s(x),Q(h)]_{\mathfrak{e}} + [Us(x),U(h)]_{\mathfrak{e}} \quad (\text{as}~~Q= U \mid_{\mathfrak{h}})\\
&= 0 \quad (\text{as}~~U~~\text{is an embedding tensor}).
\end{align*}
This proves the identities in (\ref{C}). To prove the identity (\ref{D}), we observe that 
\begin{align*}
&~~\chi(P(x), P(y)) -Q(\chi(P(x),y)) - \Phi[P(x),y]_{\mathfrak{g}} + \psi_{P(x)}\Phi(y) - \psi_{P(y)}\Phi(x) + Q (\psi_{y} \Phi(x) ) + [\Phi(x),\Phi(y)]_{\mathfrak{h}} \\
&= [sP(x),sP(y)]_{\mathfrak{e}} - s[P(x),P(y)]_{\mathfrak{g}} -Q([sP(x),s(y)]_{\mathfrak{e}} - s[P(x),y]_{\mathfrak{g}})  -Us[P(x),y]_{\mathfrak{g}} + sP[P(x),y]_{\mathfrak{g}}\\
&~~ + [sP(x),Us(y)]_{\mathfrak{e}} - [sP(x),sP(y)]_{\mathfrak{e}} -[sP(y),Us(x)]_{\mathfrak{e}} + [sP(y),sP(x)]_{\mathfrak{e}} + Q[s(y), Us(x)]_{\mathfrak{e}}\\
&~~-Q[s(y),sP(x)]_{\mathfrak{e}} + [Us(x),Us(y)]_{\mathfrak{e}} - [Us(x),sP(y)]_{\mathfrak{e}} - [sP(x), Us(y)]_{\mathfrak{e}} + [sP(x),sP(y)]_{\mathfrak{e}}\\
&= - s[P(x),P(y)]_{\mathfrak{g}} + Qs[P(x),y]_{\mathfrak{g}} -Us[P(x),y]_{\mathfrak{g}} + sP[P(x),y]_{\mathfrak{g}} +Q[s(y), Us(x)]_{\mathfrak{e}} + [Us(x),Us(y)]_{\mathfrak{e}}.
\end{align*}
The above expression vanishes as both $P$ and $U$ are embedding tensors and $Q= U|_{\mathfrak{h}}$. This completes the proof.
\end{proof}

\begin{remark}
Note that the identity (\ref{D}) can be equivalently described by 
\begin{equation}\label{D1}
\chi(P(x),P(y)) - Q(\chi(x,P(y))) -\Phi[x,P(y)]_{\mathfrak{g}} + \psi_{P(x)}\Phi(y) - \psi_{P(y)}\Phi(x)-Q(\psi_{x}\Phi(y)) + [\Phi(x),\Phi(y)]_{\mathfrak{h}}=0.
\end{equation}
\end{remark}

Let $s{'}: \mathfrak{g} \rightarrow \mathfrak{e}$ be any other section of (\ref{nab}). Let $\phi: \mathfrak{g} \rightarrow \mathfrak{h}$ be defined by $\phi(x) := s(x) - s{'}(x)$, for all $x \in \mathfrak{g}$. If $\chi{'}$, $\psi{'}$ and $\Phi{'}$ are the maps induced by the section $s{'}$, then for all $x,y \in \mathfrak{g}$ and $h \in \mathfrak{h}$, we have
\begin{align*}
\psi_{x}h-\psi_{x}{'}h &= [s(x),h]_{\mathfrak{e}} - [s{'}(x),h]_{\mathfrak{e}}= [\phi(x),h]_{\mathfrak{h}},\\
\chi(x,y) - \chi{'}(x,y) &= [s(x),s(y)]_{\mathfrak{e}} - s[x,y]_{\mathfrak{g}} - [s{'}(x),s{'}(y)]_{\mathfrak{e}} + s'[x,y]_{\mathfrak{g}}\\
                          &= [s{'}(x), (s-s{'})(y)]_{\mathfrak{e}} - [s{'}(y),(s-s{'})(x)]_{\mathfrak{e}}-(s-s{'})[x,y]_{\mathfrak{g}} + [(s-s{'})(x),(s-s{'})(y)]_{\mathfrak{h}}\\
                          &= \psi'_{x}\phi(y) - \psi'_{y}\phi(x) - \phi[x,y]_{\mathfrak{g}} + [\phi(x),\phi(y)]_{\mathfrak{h}},\\
\Phi(x) - \Phi{'}(x) &= (Us-sP)(x) - (Us{'}-s{'}P)(x)\\
                      &= U(s-s{'})(x) - (s-s{'})P(x)\\
                      &= Q\phi(x) - \phi P(x).
\end{align*}
The above discussion leads us to the following definition.

\begin{defn}
\begin{enumerate}
\item[(i)] Let $\mathfrak{g}_P$ and $\mathfrak{h}_Q$ be two averaging Lie algebras. A {\bf{non-abelian 2-cocycle}} of $\mathfrak{g}_P$ with values in $\mathfrak{h}_Q$ is a triple $(\chi, \psi, \Phi)$ of linear maps $\chi: \wedge^{2}\mathfrak{g} \rightarrow \mathfrak{h}$, $\psi: \mathfrak{g} \rightarrow \Der(\mathfrak{h})$ and $\Phi: \mathfrak{g} \rightarrow \mathfrak{h}$ satisfying the conditions (\ref{A}), (\ref{B}), (\ref{C}) and (\ref{D}).

\item[(ii)] Let $(\chi, \psi, \Phi)$ and $(\chi{'}, \psi{'}, \Phi{'})$ be two non-abelian 2-cocycles of $\mathfrak{g}_P$ with values in $\mathfrak{h}_Q$. They are said to be {\bf{equivalent}} if there exists a linear map $\phi: \mathfrak{g} \rightarrow \mathfrak{h}$ that satisfies
\begin{align}
\psi_{x}h-\psi'_{x}h =~& [\phi(x),h]_{\mathfrak{h}}, \label{E1} \\
\chi(x,y) - \chi'(x,y) =~& \psi'_{x}\phi(y) - \psi'_{y}\phi(x) - \phi[x,y]_{\mathfrak{g}} + [\phi(x),\phi(y)]_{\mathfrak{h}}, \label{E2} \\
\Phi(x) - \Phi'(x) =~&  Q\phi(x) - \phi P(x),~\text{for all}~~x, y \in \mathfrak{g}, h \in \mathfrak{h}. \label{E3}
\end{align}
\end{enumerate}
\end{defn}

Let $H^{2}_{\nab}(\mathfrak{g}_P, \mathfrak{h}_Q)$ be the set of all equivalence classes of non-abelian 2-cocycles of $\mathfrak{g}_P$ with values in $\mathfrak{h}_Q$. This is called the {\bf second non-abelian cohomology group} of the averaging Lie algebra $\mathfrak{g}_P$ with values in $\mathfrak{h}_Q$. The following result shows that the set of all equivalence classes of non-abelian extensions of $\mathfrak{g}_P$ by $\mathfrak{h}_Q$ are classified by the cohomology group $H^{2}_{\nab}(\mathfrak{g}_P, \mathfrak{h}_Q)$.

\begin{thm}\label{isom}
Let $\mathfrak{g}_P$ and $\mathfrak{h}_Q$ be two averaging Lie algebras. Then $\Ext_{\nab}(\mathfrak{g}_P, \mathfrak{h}_Q) \cong H^{2}_{\nab}(\mathfrak{g}_P,\mathfrak{h}_Q)$.
\end{thm}
\begin{proof}
Let $\mathfrak{e}_U$ and $\mathfrak{e{'}}_{U{'}}$ be two equivalent non-abelian extensions of $\mathfrak{g}_P$ by $\mathfrak{h}_Q$ (see Definition \ref{defn-nabe}(ii)). If $s: \mathfrak{g} \rightarrow \mathfrak{e}$ is a section of the map $p$, then it is easy to observe that the map $s{'} := \tau \circ s$ is a section of the map $p{'}$. Let $(\chi{'}, \psi{'}, \Phi{'})$ be the non-abelian 2-cocycle corresponding to the non-abelian extension $\mathfrak{e{'}}_{U{'}}$ and the section $s{'}$. Then we have
\begin{align*}
\chi{'}(x,y) &= [s{'}(x),s{'}(y)]_{\mathfrak{e}{'}} - s{'}[x,y]_{\mathfrak{g}}\\
              &= [\tau \circ s (x), \tau \circ s (y)]_{\mathfrak{e}{'}} - (\tau \circ s)[x,y]_{\mathfrak{g}}\\
              &= \tau ([s(x),s(y)]_{\mathfrak{e}}-s[x,y]_{\mathfrak{g}}) = \chi(x,y) \quad (\text{as}~~\tau|_{\mathfrak{h}}=\Id_{\mathfrak{h}}),
\end{align*}
\begin{align*}
\psi'_{x}(h) = [s{'}(x),h]_{\mathfrak{e}{'}}
                = [\tau \circ s(x), h]_{\mathfrak{e}{'}}
                = \tau ([s(x),h]_{\mathfrak{e}})
                = \psi_{x}(h) \quad (\text{as}~~\tau|_{\mathfrak{h}}=\Id_{\mathfrak{h}}),
\end{align*}
and
\begin{align*}
\Phi{'}(x) &= U{'}s{'}(x) -s{'}P(x) \\
            &= U^{'} (\tau \circ s(x)) - (\tau \circ s)P(x)\\
            &= \tau (Us(x) -sP(x)) = \Phi(x) \quad (\text{as}~~\tau|_{\mathfrak{h}}=\Id_{\mathfrak{h}}).
\end{align*}
Thus, we have $(\chi,\psi, \Phi)= (\chi{'},\psi{'}, \Phi{'})$. Hence they give rise to the same element in $H^{2}_{\nab}(\mathfrak{g}_P, \mathfrak{h}_Q)$. Therefore, there is a well-defined map $\Upsilon: \Ext_{\nab}(\mathfrak{g}_P, \mathfrak{h}_Q) \rightarrow H^{2}_{\nab}(\mathfrak{g}_P,\mathfrak{h}_Q)$. 

\medskip

Conversely, let $(\chi,\psi, \Phi)$ be a non-abelian 2-cocycle of $\mathfrak{g}_P$ with values in $\mathfrak{h}_Q$. Define $\mathfrak{e}:= \mathfrak{g} \oplus \mathfrak{h}$ with the bilinear skew-symmetric bracket
$$[(x,h),(y,k)]_{\mathfrak{e}} := ([x,y]_{\mathfrak{g}}, \psi_{x}k - \psi_{y}h + \chi(x,y) +[h,k]_{\mathfrak{h}}),~~\text{for}~~(x,h), (y,k) \in \mathfrak{e}.$$ It has been observed in \cite{fregier} (using the conditions (\ref{A}) and (\ref{B})) that the bracket $[~,~]_{\mathfrak{e}}$ satisfies the Jacobi identity. In other words, $(\mathfrak{e}, [~,~]_{\mathfrak{e}})$ is a Lie algebra. Further, we define a map $U: \mathfrak{e} \rightarrow \mathfrak{e}$ by 
$$U(x,h) := (P(x), Q(h) + \Phi(x)),~\text{for all}~~(x,h) \in \mathfrak{e}.$$ 
Then we have
\begin{align*}
&[U(x,h), U(y,k)]_{\mathfrak{e}}\\
&= \left[ (P(x), Q(h) + \Phi(x)),(P(y), Q(k) + \Phi(y))  \right]_\mathfrak{e} \\
                                &= \big( [P(x),P(y)]_{\mathfrak{g}}, ~\psi_{P(x)}Q(k) + \psi_{P(x)}\Phi(y) - \psi_{P(y)}Q(h) - \psi_{P(y)}\Phi(x) + \chi(P(x),P(y))\\
                                & \qquad \qquad \qquad \qquad  + [Q(h),Q(k)]_{\mathfrak{h}} + [Q(h), \Phi(y)]_{\mathfrak{h}} + [\Phi(x), Q(k)]_{\mathfrak{h}} + [\Phi(x), \Phi(y)]_{\mathfrak{h}} \big)\\
                                &=\big(P[P(x),y]_{\mathfrak{g}},~ Q(\psi_{P(x)}k - \psi_{y} Q(h) - \psi_{y} \Phi(x) + \chi(P(x),y) + [Q(y),k]_{\mathfrak{h}} + [\Phi(x),k]_{\mathfrak{h}})\\
                                &\qquad \qquad \qquad \qquad + \Phi[P(x),y]_{\mathfrak{g}} \big) \quad (\text{by}~~(\ref{C}))\\
                                &= U \big([P(x),y]_{\mathfrak{g}}, ~ \psi_{P(x)}k - \psi_{y} Q(h) - \psi_{y} \Phi(x) + \chi(P(x),y) + [Q(y),k]_{\mathfrak{h}} + [\Phi(x),k]_{\mathfrak{h}} \big)\\
                                &= U \big([(P(x), Q(h) + \Phi(x)), (y,k)]_{\mathfrak{e}} \big)\\
                                &= U([U(x,h), (y,k)]_{\mathfrak{e}}).
\end{align*}
Therefore, the map $U: \mathfrak{e} \rightarrow \mathfrak{e}$ is an averaging operator on the Lie algebra $\mathfrak{e}$. In other words, $\mathfrak{e}_U$ is an averaging Lie algebra. Further, it is easy to see that 
\begin{align*}
\xymatrix{0 \ar[r] & \mathfrak{h}_Q \ar[r]^{i} & \mathfrak{e}_U \ar[r]^{p}  &  \mathfrak{g}_P \ar[r] & 0}
\end{align*}
is a non-abelian extension of the averaging Lie algebra $\mathfrak{g}_P$ by $\mathfrak{h}_Q$, where $i(h) = (0, h)$ and $p(x,h)=x$, for all $(x,h) \in \mathfrak{e}$ and $h \in \mathfrak{h}$.

Next, let $(\chi, \psi, \Phi)$ and $(\chi{'}, \psi{'}, \Phi{'})$ be two equivalent non-abelian 2-cocycles, i.e., there exists a linear map $\phi: \mathfrak{g} \rightarrow \mathfrak{h}$ such that the identities (\ref{E1}), (\ref{E2}) and (\ref{E3}) hold. Let $\mathfrak{e}{'}_{U{'}}$ be the averaging Lie algebra induced by the 2-cocycle $(\chi{'}, \psi{'}, \Phi{'})$. Note that the Lie algebra $\mathfrak{e}{'}$ is the vector space $\mathfrak{g} \oplus \mathfrak{h}$ with the bracket 
$$[(x,h),(y,k)]_{\mathfrak{e}{'}} := ([x,y]_{\mathfrak{g}},~ \psi'_{x}k - \psi'_{y}h + \chi{'}(x,y) + [h,k]_{\mathfrak{h}}),~\text{for}~~(x,h), (y,k) \in \mathfrak{e}{'}.$$
Moreover, the map $U{'}: \mathfrak{e}{'} \rightarrow \mathfrak{e}{'}$ is given by $U{'}(x,h) := (P(x), Q(h) + \Phi{'}(x))$, {for all} $(x,h) \in \mathfrak{e}'.$
We now define a map $\tau: \mathfrak{g} \oplus \mathfrak{h} \rightarrow \mathfrak{g} \oplus \mathfrak{h}$ by $\tau(x,h) := (x, h + \phi(x))$, for all $(x,h) \in \mathfrak{g} \oplus \mathfrak{h}$. Then by a straightforward calculation, we have $\tau [(x,h),(y,k)]_{\mathfrak{e}} = [\tau (x,h), \tau (y,k)]_{\mathfrak{e}{'}}$. Further, 
\begin{align*}
(U \circ \tau) (x,h) &= U{'}(x,h + \phi(x))\\
                       &= (P(x), Q(h)+ Q(\phi(x)) + \Phi{'}(x))\\
                       &= (P(x), Q(h) + \phi(P(x)) + \Phi(x))  \quad (\text{by}~~(\ref{E3}))\\
                       &= \tau (P(x), Q(h) + \Phi(x))
                       = (\tau \circ U)(x,h).
\end{align*}
Hence the map $\tau: \mathfrak{e}_{U} \rightarrow \mathfrak{e}{'}_{U{'}}$ defines an equivalence between two non-abelian extensions. Therefore, we obtain a well-defined map 
$\Lambda: H^{2}_{\nab}(\mathfrak{g}_P, \mathfrak{h}_Q) \rightarrow \Ext_{\nab}(\mathfrak{g}_P, \mathfrak{h}_Q)$. Finally, it is straightforward to verify that the maps $\Upsilon$ and $\Lambda$ are inverses to each other. This completes the proof.
\end{proof}

\section{Automorphisms of averaging Lie algebras and the Wells map}\label{sec5}

Given a non-abelian extension of averaging Lie algebras, we consider the inducibility of a pair of averaging Lie algebra automorphisms. We show that the corresponding obstruction (which lies in the second non-abelian cohomology group) can be described as the image of a suitable Well map. Finally, we construct the Well short exact sequence that connects various automorphism groups and the second non-abelian cohomology group. 

Let $\mathfrak{g}_P$ and $\mathfrak{h}_Q$ be two averaging Lie algebras and 
\begin{align*}
\xymatrix{0 \ar[r] & \mathfrak{h}_Q \ar[r]^{i} & \mathfrak{e}_U \ar[r]^{p}  &  \mathfrak{g}_P \ar[r] & 0}
\end{align*} be a non-abelian extension of $\mathfrak{g}_P$ by $\mathfrak{h}_Q$. Let $\Aut_{\mathfrak{h}}(\mathfrak{e}_U)$ be the set of all averaging Lie algebra automorphisms $\gamma \in \Aut(\mathfrak{e}_{U})$ that satisfies $\gamma|_{\mathfrak{h}} \subset \mathfrak{h}$. For such a $\gamma \in \Aut_{\mathfrak{h}}(\mathfrak{e}_U)$, we have $\gamma|_{\mathfrak{h}} \in \Aut(\mathfrak{h}_Q)$. For any section $s: \mathfrak{g} \rightarrow \mathfrak{e}$ of the map $p$, we define a map $\overline{\gamma}: \mathfrak{g} \rightarrow \mathfrak{g}$ by $\overline{\gamma}(x) := p\gamma s(x)$, for $x \in \mathfrak{g}$. It is easy to verify that the map $\overline{\gamma}$ is independent of the choice of the section $s$ and $\overline{\gamma}$ is a bijection on $\mathfrak{g}$. Further, for any $x,y \in \mathfrak{g}$, we have
\begin{align*}
\overline{\gamma}([x,y]_{\mathfrak{g}}) = p\gamma(s[x,y]_{\mathfrak{g}}) &= p\gamma([s(x),s(y)]_{\mathfrak{e}}-\chi(x,y))\\
                                                                   &= p\gamma[s(x),s(y)]_{\mathfrak{e}} \quad (\text{as}~~\gamma|_{\mathfrak{h}} \subset \mathfrak{h}~~\text{and}~~p|_{\mathfrak{h}}=0)\\
                                                                   &= [p\gamma s(x),p\gamma s(y)]_{\mathfrak{g}}
                                                                =[\overline{\gamma}(x),\overline{\gamma}(y)]_{\mathfrak{g}}
\end{align*}
and 
\begin{align*}
(P \overline{\gamma} - \overline{\gamma}P)(x) &= (Pp\gamma s - p\gamma s P)(x)\\
                                    &= (pU\gamma s - p \gamma s P)(x)\\
                                    &= p\gamma (Us -sP)(x)
                                    = 0 \quad (\text{as}~~\gamma|_{\mathfrak{h}} \subset \mathfrak{h}~~\text{and}~~p|_{\mathfrak{h}}=0).
\end{align*}
This shows that the map $\overline{\gamma}: \mathfrak{g} \rightarrow \mathfrak{g}$ is an automorphism of the averaging Lie algebra $\mathfrak{g}_P$. In other words, $\overline{\gamma} \in \Aut(\mathfrak{g}_P)$. Therefore, we obtain a group homomorphism 
$$\Pi: \Aut_{\mathfrak{h}}(\mathfrak{e}_U) \rightarrow \Aut(\mathfrak{h}_Q) \times \Aut(\mathfrak{g}_P)~~\text{defined by}~~\Pi(\gamma) := (\gamma|_{\mathfrak{h}}, \overline{\gamma}),~\text{for all}~~\gamma \in \Aut_{\mathfrak{h}}(\mathfrak{e}_U).$$

\begin{defn}
A pair $(\beta, \alpha) \in \Aut(\mathfrak{h}_Q) \times \Aut(\mathfrak{g}_P)$ of averaging Lie algebra automorphisms is said to be {\bf{inducible}} if the pair $(\beta, \alpha)$ lies in the image of $\Pi$, i.e., $(\beta,\alpha)$ is inducible if there exists $\gamma \in \Aut_{\mathfrak{h}}(\mathfrak{e}_U)$ such that $\gamma|_{\mathfrak{h}}=\beta$ and $\overline{\gamma} = \alpha$.
\end{defn}

Our main aim in this section is to find a necessary and sufficient condition for a pair of averaging Lie algebra automorphisms $(\beta, \alpha) \in \Aut(\mathfrak{h}_Q) \times \Aut(\mathfrak{g}_P)$ to be inducible. The main theorem of this section will be stated in terms of the Wells map in the context of averaging Lie algebras.

Let $0 \rightarrow \mathfrak{h}_{Q} \xrightarrow{i}{\mathfrak{e}_{U}} \xrightarrow{p}{\mathfrak{g}_{P}} \rightarrow 0$ be a non-abelian extension of averaging Lie algebras. For any fixed section $s: \mathfrak{g} \rightarrow \mathfrak{e}$ of the map $p$, let $(\chi, \psi, \Phi)$ be the corresponding non-abelian 2-cocycle. Given any pair $(\beta, \alpha) \in \Aut(\mathfrak{h}_Q) \times \Aut(\mathfrak{g}_P)$ of averaging Lie algebra automorphisms, we define a new triple $(\chi_{(\beta, \alpha)}, \psi_{(\beta, \alpha)}, \Phi_{(\beta, \alpha)})$ of linear maps $\chi_{(\beta, \alpha)}: \wedge^{2}\mathfrak{g} \rightarrow \mathfrak{h}$, $\psi_{(\beta, \alpha)}: \mathfrak{g} \rightarrow \Der(\mathfrak{h})$ and $\Phi_{(\beta, \alpha)}: \mathfrak{g} \rightarrow \mathfrak{h}$ by 
\begin{equation}\label{N}
\chi_{(\beta, \alpha)}(x,y) := \beta \circ \chi(\alpha^{-1}(x), \alpha^{-1}(y)),~~(\psi_{(\beta, \alpha)})_{x}h := \beta (\psi_{\alpha^{-1}(x)}\beta^{-1}(h))~~\text{and}~~\Phi_{(\beta, \alpha)}(x) := \beta \Phi (\alpha^{-1}(x)),
\end{equation}
for $x,y \in \mathfrak{g}$ and $h \in \mathfrak{h}$. Then we have the following result.
\begin{lemma}
The triple $(\chi_{(\beta, \alpha)}, \psi_{(\beta, \alpha)}, \Phi_{(\beta, \alpha)})$ is a non-abelian $2$-cocycle.
\end{lemma}
\begin{proof}
The triple $(\chi, \psi, \Phi)$ is a non-abelian $2$-cocycle implies that the identities (\ref{A}), (\ref{B}), (\ref{C}) and (\ref{D}) hold. In these identities, if we replace $x,y,h$ by $\alpha^{-1}(x),\alpha^{-1}(y),\beta^{-1}(h)$ respectively, we simply get the non-abelian $2$-cocycle conditions for the triple $(\chi_{(\beta, \alpha)}, \psi_{(\beta, \alpha)}, \Phi_{(\beta, \alpha)})$. For example, it follows from (\ref{C}) that
\begin{align*}
\beta(\psi_{P \alpha^{-1}(x)}Q\beta^{-1}(h)) = \beta Q(\psi_{P \alpha^{-1}(x)}\beta^{-1}(h)) + \beta Q [\Phi \alpha^{-1}(x), \beta^{-1}(h)]_{\mathfrak{h}} - \beta[\Phi \alpha^{-1}(x), Q \beta^{-1}(h)]_{\mathfrak{h}}\nonumber \\
                                             = \beta Q(\psi_{\alpha^{-1}(x)} Q\beta^{-1}(h)) - \beta[\Phi \alpha^{-1}(x), Q \beta^{-1}(h)]_{\mathfrak{h}}. \nonumber
\end{align*}
This can be written as (using (\ref{N}))
\begin{align*}
(\psi_{(\beta, \alpha)})_{P(x)}Q(h) = Q((\psi_{(\beta, \alpha)})_{P(x)}h) ~+~& Q[\Phi_{(\beta, \alpha)}(x), h]_{\mathfrak{h}} - [\Phi_{(\beta, \alpha)}(x), Q(h)]_{\mathfrak{h}}\nonumber \\
                                    =~& Q((\psi_{(\beta, \alpha)})_{x}Q(h)) - [\Phi_{(\beta, \alpha)}(x), Q(h)]_{\mathfrak{h}}.\nonumber
\end{align*}
This shows that the identity (\ref{C}) is also holds for the triple $(\chi_{(\beta, \alpha)}, \psi_{(\beta, \alpha)}, \Phi_{(\beta, \alpha)})$.
\end{proof}

Note that the non-abelian $2$-cocycle $(\chi, \psi, \Phi)$ and therefore  $(\chi_{(\beta, \alpha)}, \psi_{(\beta, \alpha)}, \Phi_{(\beta, \alpha)})$ depends on the chosen section $s$. We now define a map 
$\mathcal{W}: \Aut(\mathfrak{h}_Q) \times \Aut(\mathfrak{g}_P) \rightarrow H^{2}_{\nab}(\mathfrak{g}_P, \mathfrak{h}_Q)$ by 
\begin{equation}\label{W}
\mathcal{W}((\beta, \alpha)) = [  (\chi_{(\beta, \alpha)}, \psi_{(\beta, \alpha)}, \Phi_{(\beta, \alpha)}) -(\chi, \psi, \Phi) ],
\end{equation}
{the equivalence class of} $(\chi_{(\beta, \alpha)}, \psi_{(\beta, \alpha)}, \Phi_{(\beta, \alpha)}) -(\chi, \psi, \Phi).$
The map $\mathcal{W}$ is called the {\bf{Wells map}}. Note that the Wells map may not be a group homomorphism.

\begin{prop}
The Wells map $\mathcal{W}$ does not depend on the chosen section.
\end{prop}
\begin{proof}
Let $s{'}$ be any other section of the map $p$ and let $(\chi{'}, \psi{'}, \Phi{'})$ be the corresponding non-abelian $2$-cocycle. We have seen that the non-abelian $2$-cocycles $(\chi, \psi, \Phi)$ and $(\chi{'}, \psi{'}, \Phi{'})$ are equivalent by the map $\phi := s- s{'}$. Using this, it is easy to verify that the non-abelian $2$-cocycles $(\chi_{(\beta, \alpha)}, \psi_{(\beta, \alpha)}, \Phi_{(\beta, \alpha)})$ and $(\chi'_{(\beta, \alpha)}, \psi'_{(\beta, \alpha)}, \Phi{'}_{(\beta, \alpha)})$ are equivalent by the map $\beta \phi \alpha^{-1}$. Combining these results, we observe that the $2$-cocycles $(\chi_{(\beta, \alpha)}, \psi_{(\beta, \alpha)}, \Phi_{(\beta, \alpha)}) -(\chi, \psi, \Phi)$ and $(\chi'_{(\beta, \alpha)}, \Psi'_{(\beta, \alpha)}, \Phi'_{(\beta, \alpha)}) -(\chi{'}, \psi{'}, \Phi{'})$ are equivalent by the map $\beta \phi \alpha^{-1} - \phi$. Therefore, their corresponding equivalence calsses in $H^{2}_{\nab}(\mathfrak{g}_P, \mathfrak{h}_Q)$ are same. In other words, the map $\mathcal{W}$ does not depend on the chosen section.
\end{proof}

We are now in a position to prove the main result of this section. 
\begin{thm}\label{main}
Let $0 \rightarrow \mathfrak{h}_{Q} \xrightarrow{i}{\mathfrak{e}_{U}} \xrightarrow{p}{\mathfrak{g}_{P}} \rightarrow 0$
be a non-abelian extension of averaging Lie algebras. A pair $(\beta, \alpha) \in \Aut (\mathfrak{h}_Q) \times \Aut(\mathfrak{g}_P)$ of averaging Lie algebra automorphisms is inducible if and only if $\mathcal{W}((\beta, \alpha))=0$.
\end{thm}
\begin{proof}
Let $(\beta, \alpha)  \in \Aut (\mathfrak{h}_Q) \times \Aut(\mathfrak{g}_P)$ be indicible. Therefore, there exists an automorphism $\gamma \in \Aut_{\mathfrak{h}}(\mathfrak{e}_{U})$ such that $\gamma|_{\mathfrak{h}} = \beta$ and $p \gamma s = \alpha$ (for any section $s$). For any $x \in \mathfrak{g}$, we observe that
$$p(\gamma s - s \alpha)(x) = \alpha(x) - \alpha (x) =0.$$
Therefore, $(\gamma s - s \alpha)(x) \in \Ker(p)= \Im(i) \cong \mathfrak{h}$. We define a map $\phi: \mathfrak{g} \rightarrow \mathfrak{h}$ by $\phi(x) := (\gamma s - s \alpha) \alpha^{-1}(x) = \gamma s \alpha^{-1}(x) - s(x)$, for $x \in \mathfrak{g}$. Then we have
\begin{align*}
(\psi_{(\beta, \alpha)})_{x} h - \psi_{x} h &= \beta (\psi_{\alpha^{-1}(x)}\beta^{-1}(h)) -\psi_{x} h\\
                                            &= \beta [s \alpha^{-1}(x), \beta^{-1}(h)]_{\mathfrak{e}} - [s(x),h]_{\mathfrak{e}}\\
                                            &= [\gamma s \alpha^{-1}(x),h]_{\mathfrak{e}} - [s(x),h]_{\mathfrak{e}} \quad (\text{as}~~\beta = \gamma|_{\mathfrak{h}})\\
                                            &= [\phi(x),h]_{\mathfrak{h}},~\text{for any}~~x \in \mathfrak{g}~~\text{and}~~h \in \mathfrak{h}.
\end{align*}
Similarly, by direct calculations, we observe that 
\begin{align*}
\chi_{(\beta, \alpha)}(x,y) - \chi(x,y)=~& \psi_{x} \phi(y) - \psi_{y} \phi(x) - \phi ([x,y]_{\mathfrak{g}}) + [\phi(x), \phi(y)]_{\mathfrak{h}},\\
\Phi_{(\beta, \alpha)}(x) - \Phi (x)=~& Q \phi(x) - \phi P(x),~~\text{for any}~~x,y \in \mathfrak{g}.
\end{align*}
Here $(\chi, \psi, \Phi)$ is the non-abelian $2$-cocycle induced by any fixed section $s$. It follows from the above observation that the non-abelian $2$-cocycles $(\chi_{(\beta, \alpha)}, \psi_{(\beta, \alpha)}, \Phi_{(\beta, \alpha)})$ and $(\chi, \psi, \Phi)$ are equivalent by the map $\phi = \gamma s \alpha^{-1} -s$. Hence we have
$$\mathcal{W}((\beta, \alpha)) := [(\chi_{(\beta, \alpha)}, \psi_{(\beta, \alpha)}, \Phi_{(\beta, \alpha)})  -  (\chi, \psi, \Phi)]=0.$$

Conversely, let $(\beta, \alpha) \in \Aut(\mathfrak{h}_Q) \times \Aut(\mathfrak{g}_P)$ be a pair of averaging Lie algebra automorphisms such that $\mathcal{W}((\beta, \alpha))=0$. Let $s: \mathfrak{g} \rightarrow \mathfrak{e}$ be any section of the map $p$ and $(\chi, \psi, \Phi)$ be the non-abelian $2$-cocycle induced by the section $s$. Since $\mathcal{W}((\beta, \alpha))=0$, it follows that the non-abelian $2$-cocycles $(\chi_{(\beta, \alpha)}, \psi_{(\beta, \alpha)}, \Phi_{(\beta, \alpha)})$ and $(\chi, \psi, \Phi)$ are equivalent, say by the map $\phi: \mathfrak{g} \rightarrow \mathfrak{h}$. Note that, $s$ is a section of the map $p$ implies that any element $e \in \mathfrak{e}$ can be written as $e = h + s(x)$, for some $h \in \mathfrak{h}$ and $x \in \mathfrak{g}$. We now define a map $\gamma : \mathfrak{e} \rightarrow \mathfrak{e}$ by 
$$\gamma(e) = \gamma (h + s(x)):=(\beta(h) + \phi \alpha(x)) + s(\alpha(x)),~\text{for}~~e = h+ s(x) \in \mathfrak{e}.$$

\medskip

\noindent \underline{\bf{Step I.}} ($\gamma$ is bijective) Suppose $\gamma(h+s(x))=0$. Then it follows that $s(\alpha(x))=0$. As both $s$ and $\alpha$ are injective maps, we have $x=0$. Using this in the definition of $\gamma$, we have $\beta(h) =0$, which implies $h=0$. This proves that $\gamma$ is injective. Finally, let $e = h + s(x) \in \mathfrak{e}$ be any arbitrary element. We consider the element $e{'}= (\beta^{-1}(h) - \beta^{-1} \phi(x)) + s (\alpha^{-1}(x)) \in \mathfrak{e}$. Then we have $$\gamma(e{'})= (h - \phi(x) +\phi(x)) + s(x) = h + s(x) =e.$$
This shows that $\gamma$ is also surjective. Hence $\gamma$ is bijective. 

\medskip

\noindent \underline{\bf{Step II.}} ($\gamma: \mathfrak{e}_{U} \rightarrow \mathfrak{e}_{U}$ is an automorphism of averaging Lie algebras) Take any two elements $e_1 = h_1 + s(x_1)$ and $e_2 = h_2 + s(x_2)$ of the vector space $\mathfrak{e}$. We have 
\begin{align*}
[\gamma(e_1), \gamma(e_2)]_{e} &= [\beta(h_1) + \phi \alpha (x_1) + s(\alpha(x_1)), \beta(h_2) + \phi \alpha (x_2) + s(\alpha(x_2))]_{\mathfrak{e}}\\
                               &= [\beta(h_1), \beta(h_2)]_{\mathfrak{e}} + \underbrace{[\beta(h_1), \phi \alpha (x_2)]_{\mathfrak{e}}}_{A} + \underbrace{[\beta(h_1), s(\alpha(x_2))]_{\mathfrak{e}}}_{B} + \underbrace{[\phi \alpha (x_1),\beta(h_2)]_{\mathfrak{e}}}_{C}\\
                               &~~ + \underbrace{[\phi \alpha (x_1),\phi \alpha (x_2)]_{\mathfrak{e}}}_{D} + \underbrace{[\phi \alpha (x_1),s(\alpha(x_2))]_{\mathfrak{e}}}_{E} + \underbrace{[s(\alpha(x_1)), \beta(h_2)]_{\mathfrak{e}}}_{F} + \underbrace{[s(\alpha(x_1)), \phi \alpha (x_2)]_{\mathfrak{e}}}_{G}\\
                               &~~+ \underbrace{[s(\alpha(x_1)),s(\alpha(x_2))]_{\mathfrak{e}}}_{H}\\
                               &= [\beta(h_1), \beta(h_2)]_{\mathfrak{e}} \underbrace{- \beta(\psi_{x_2}h_1) + \psi_{\alpha(x_2)}\beta(h_1)}_{A} \underbrace{- \psi_{\alpha(x_2)}\beta(h_1)}_{B} + \underbrace{\beta(\psi_{x_1}h_2) - \psi_{\alpha(x_1)}\beta(h_2)}_{C}\\
                               &~~+ \underbrace{\beta(\chi(x_1,x_2))- \chi(\alpha(x_1), \alpha(x_2))-\psi_{\alpha(x_1)}\phi \alpha(x_2) + \psi_{\alpha(x_2)} \phi \alpha(x_1) + \phi \alpha([x_1,x_2]_{\mathfrak{g}})}_{D}\\
                               &~~ \underbrace{ - \psi_{\alpha(x_2)}\phi \alpha(x_1)}_{E} + \underbrace{\psi_{\alpha(x_1)}\beta(h_2)}_{F} + \underbrace{\psi_{\alpha(x_1)} \phi \alpha(x_2)}_{G} + \underbrace{\chi(\alpha(x_1), \alpha(x_2))+ s[\alpha(x_1), \alpha (x_2)]_{\mathfrak{g}}}_{H}\\
                               &= \beta ([h_1,h_2]_{\mathfrak{h}} - \psi_{x_2}h_1 + \psi_{x_1} h_2 + \chi(x_1, x_2)) + \lambda([x_1,x_2]_{\mathfrak{g}}) + s[\alpha(x_1), \alpha(x_2)]_{\mathfrak{g}}\\
                               &= \beta ([h_1,h_2]_{\mathfrak{h}} + [h_1, s(x_2)]_{\mathfrak{e}} + [s(x_1),h_2]_{\mathfrak{e}} + \chi(x_1, x_2)) + \lambda([x_1,x_2]_{\mathfrak{g}}) + s \alpha([x_1, x_2]_{\mathfrak{g}})\\
                               &= \gamma ([h_1,h_2]_{\mathfrak{h}} + [h_1, s(x_2)]_{\mathfrak{e}} + [s(x_1),h_2]_{\mathfrak{e}} + \chi(x_1, x_2) + s[x_1,x_2]_{\mathfrak{g}})\\
                               &= \gamma ([h_1,h_2]_{\mathfrak{h}} + [h_1, s(x_2)]_{\mathfrak{e}} + [s(x_1),h_2]_{\mathfrak{e}} + [s(x_1),s(x_2)]_{\mathfrak{e}})\\
                               &= \gamma ([h_1 + s(x_1), h_2 + s(x_2)]_{\mathfrak{e}})\\
                               &= \gamma ([e_1, e_2]_{\mathfrak{e}}).
\end{align*}
This shows that $\gamma$ preserves the Lie bracket. It is straightforward to verify that $\gamma \circ U = U \circ \gamma$. Hence $\gamma : \mathfrak{e}_{U} \rightarrow \mathfrak{e}_{U}$ is an automorphism of averaging Lie algebras.

\medskip

\noindent \underline{\bf{Step III.}} For any $h \in \mathfrak{h}$ and $x \in \mathfrak{g}$, we observe that 
$$\gamma(h)= \gamma (h + s(0)) = \beta (h)~~\text{and}$$
$$(p \gamma s)(x) = p \gamma (0 + s(x)) = p (\phi \alpha(x) + s(\alpha(x)))= ps(\alpha(x))= \alpha(x).$$
This shows that $\Pi(\gamma) = (\gamma|_{\mathfrak{h}}, p\gamma s) = (\beta, \alpha)$. Hence the pair $(\beta, \alpha)$ is inducible.
\end{proof}

\begin{remark}
It follows from the previous theorem that $\mathcal{W}((\beta, \alpha))$ is an obstruction for the inducibility of the pair $(\beta, \alpha)$. Therefore, the images of the Wells map are the obstructions for the inducibility of pair of automorphisms in $\Aut(\mathfrak{h}_Q) \times \Aut(\mathfrak{g}_P)$. Further, it follows that if the Wells map $\mathcal{W}$ is the trivial map, then any pair of averaging Lie algebra automorphisms in $\Aut(\mathfrak{h}_Q) \times \Aut(\mathfrak{g}_P)$ is inducible.
\end{remark}

Note that the Wells map defined in (\ref{W}) generalizes the well-known Wells map from the Lie algebra context. In the Lie algebra case, the Wells map fits into an exact sequence (known as the Wells exact sequence). We will now generalize the Wells exact sequence in the context of averaging Lie algebras. 

\medskip

Let $0 \rightarrow \mathfrak{h}_{Q} \xrightarrow{i}{\mathfrak{e}_{U}} \xrightarrow{p}{\mathfrak{g}_{P}} \rightarrow 0$
be a non-abelian extension of averaging Lie algebras. We define a subgroup $\Aut_{\mathfrak{h}}^{\mathfrak{h},\mathfrak{g}}(\mathfrak{e}_{U}) \subset \Aut_{\mathfrak{h}}(\mathfrak{e}_U)$ 
by $$\Aut_{\mathfrak{h}}^{\mathfrak{h},\mathfrak{g}}(\mathfrak{e}_{U}) := \lbrace \gamma \in \Aut_{\mathfrak{h}}(\mathfrak{e}_{U}) \mid \Pi (\gamma) = (\Id_{\mathfrak{h}}, \Id_{\mathfrak{g}})   \rbrace.$$

\begin{thm}\label{w-e-s}
With the above notation, there is an exact sequence 
\begin{align}\label{star}
1 \rightarrow \Aut_{\mathfrak{h}}^{\mathfrak{h},\mathfrak{g}}(\mathfrak{e}_{U}) \xrightarrow{\iota}  \Aut_{\mathfrak{h}}(\mathfrak{e}_{U})  \xrightarrow{\Pi}   \Aut(\mathfrak{h}_Q) \times \Aut(\mathfrak{g}_P) \xrightarrow{\mathcal{W}}  H^{2}_{\nab}(\mathfrak{g}_P, \mathfrak{h}_Q).
\end{align}

\end{thm}
 \begin{proof}
The sequence (\ref{star}) is exact at the first term as the inclusion map $\iota :  \Aut_{\mathfrak{h}}^{\mathfrak{h},\mathfrak{g}}(\mathfrak{e}_{U}) \rightarrow \Aut_{\mathfrak{h}}(\mathfrak{e}_U)$ is injective. Next, we shall show that the sequence is exact at the second term. Let $\gamma \in \Ker(\Pi)$. Then we have $\gamma|_{\mathfrak{h}} = \Id_{\mathfrak{h}}$ and $p \gamma s = \Id_{\mathfrak{g}}$ (for any give section $s$). Therefore, $\gamma \in \Aut_{\mathfrak{h}}^{\mathfrak{h},\mathfrak{g}}(\mathfrak{e}_{U})$. Conversely, if $\gamma \in \Aut_{\mathfrak{h}}^{\mathfrak{h},\mathfrak{g}}(\mathfrak{e}_{U})$, then $\Pi (\gamma) = (\gamma|_{\mathfrak{h}}, p \gamma s) = (\Id_{\mathfrak{h}}, \Id_{\mathfrak{g}})$, i.e., $\gamma \in \Ker(\Pi)$. This proves that $\Ker(\Pi) = \Aut_{\mathfrak{h}}^{\mathfrak{h},\mathfrak{g}}(\mathfrak{e}_{U}) = \Im(\iota)$.

Finally, we take an element $(\beta, \alpha) \in \Ker(\mathcal{W})$. Then it follows from Theorem \ref{main} that the pair $(\beta, \alpha)$ is inducible. Hence there exists an automorphism $\gamma \in \Aut_{\mathfrak{h}}(\mathfrak{e}_{U})$ such that $\Pi (\gamma) = (\beta, \alpha)$. This shows that $(\beta, \alpha) \in \Im(\Pi)$. Conversely, let $(\beta, \alpha) \in \Im(\Pi)$, i.e., $(\beta, \alpha)$ is inducible. Then again by Theorem \ref{main}, we have $\mathcal{W}((\beta, \alpha)) =0$. Hence $(\beta, \alpha) \in \Ker(\mathcal{W})$ and therefore, we obtain $\Ker(\mathcal{W}) = \Im (\Pi)$. This shows that the sequence (\ref{star}) is also exact in the third term and this completes the proof.
\end{proof}

\section{Abelian extensions of averaging Lie algebras: a particular case}\label{sec6}

In this final section, we consider abelian extensions of averaging Lie algebras as a particular case of non-abelian extensions. Therefore, we obtain new results for abelian extensions of averaging Lie algebras that were not studied in the literature.

Let $\mathfrak{g}_P$ be an averaging Lie algebra and $\mathfrak{h}_Q$ be a representation of it (see Definition \ref{rep}). Consider $\mathfrak{h}_Q$ as an averaging Lie algebra, where $\mathfrak{h}$ is equipped with the abelian Lie bracket. Let 
\begin{align}\label{ab}
\xymatrix{0 \ar[r] & \mathfrak{h}_Q \ar[r]^{i} & \mathfrak{e}_U \ar[r]^{p}  &  \mathfrak{g}_P \ar[r] & 0}
\end{align}
be a short exact sequence of averaging Lie algebras. Note that, for any section $s: \mathfrak{g} \rightarrow \mathfrak{e}$ of the map $p$, there is an induced $\mathfrak{g}$-representation on $\mathfrak{h}$ with the action map $\psi: \mathfrak{g} \rightarrow \End(\mathfrak{h})$ given by $\psi_{x}h := [s(x),h]_{\mathfrak{e}}$, for $x \in \mathfrak{g}, h \in \mathfrak{h}$. It is easy to verify that this representation does not depend on the choice of section. Moreover, with this induced $\mathfrak{g}$-representation, $\mathfrak{h}_Q$ becomes a representation of the averaging Lie algebra $\mathfrak{g}_P$.

An extension (\ref{ab}) is said to be an {\bf abelian extension} if this new representation on $\mathfrak{h}_Q$ coincides with the given one. Equivalence between two abelian extensions can be defined similarly as in the case of non-abelian extensions. Let $\Ext_{\ab}(\mathfrak{g}_P, \mathfrak{h}_Q)$ be the set of all equivalence classes of abelian extensions of $\mathfrak{g}_P$ by the given representation $\mathfrak{h}_Q$. Then Theorem \ref{isom} can be rephrased by the following. 

\begin{thm}
Let $\mathfrak{g}_P$ be an averaging Lie algebra and $\mathfrak{h}_Q$ be a representation of it. Then there is a bijection $\Ext_{\ab}(\mathfrak{g}_P, \mathfrak{h}_Q) \cong H^{2}_{\ALie}(\mathfrak{g}_P, \mathfrak{h}_Q)$.
\end{thm}
\begin{proof}
Note that an abelian extension is a (non-abelian) extension in which $\mathfrak{h}$ has abelian Lie algebra structure and the induced $\mathfrak{g}$-representation on $\mathfrak{h}$ coincides with the prescribed one. On the other hand, if we impose these conditions on a non-abelian $2$-cocycle $(\chi, \psi, \Phi)$ (i.e., conditions (\ref{A}), (\ref{B}),  (\ref{C}), (\ref{D})), we simply get that $\psi$ is the prescribed $\mathfrak{g}$-representation on $\mathfrak{h}$ and $(\chi, \Phi)$ is a $2$-cocycle in the cochain complex $\lbrace C^{\bullet}_{\ALie} (\mathfrak{g}_P, \mathfrak{h}_Q), \delta _{\ALie} \rbrace$ defined in Section \ref{sec2}. Thus, abelian extensions give rise to $2$-cocycles in $\lbrace C^{\bullet}_{\ALie} (\mathfrak{g}_P, \mathfrak{h}_Q), \delta _{\ALie} \rbrace$. It is also similar to verifying that equivalent abelian extensions correspond to cohomologous $2$-cocycles. Hence we get the bijection $\Ext_{\ab}(\mathfrak{g}_P, \mathfrak{h}_Q) \cong H^{2}_{\ALie}(\mathfrak{g}_P, \mathfrak{h}_Q)$.
\end{proof}

Let $\mathfrak{g}_P$ be an averaging Lie algebra, $\mathfrak{h}_Q$ be a representation and  $0 \rightarrow \mathfrak{h}_{Q} \xrightarrow{i}{\mathfrak{e}_{U}} \xrightarrow{p}{\mathfrak{g}_{P}} \rightarrow 0$
be an abelian extension of $\mathfrak{g}_P$ by $\mathfrak{h}_Q$. Let $\psi: \mathfrak{g} \rightarrow \End(\mathfrak{h})$ denotes the $\mathfrak{g}$-representation on $\mathfrak{h}$. We define a subgroup $C_{\psi} \subset \Aut(\mathfrak{h}_Q) \times \Aut(\mathfrak{g}_P)$ by 
$$C_{\psi} := \lbrace (\beta, \alpha) \in \Aut(\mathfrak{h}_Q) \times \Aut(\mathfrak{g}_P) \mid \beta (\psi_{x}h)= \psi_{\alpha(x)} \beta (h),~\text{for all}~~x \in \mathfrak{g}, h \in \mathfrak{h} \rbrace.$$
The space $C_{\psi}$ is called the space of compatible pairs of automorphisms. Note that, if $s: \mathfrak{g} \rightarrow \mathfrak{e}$ is a section of $p$ and $(\chi, \Phi)$ is the $2$-cocycle induced by $s$, then for any $(\beta, \alpha) \in \Aut(\mathfrak{h}_Q) \times \Aut(\mathfrak{g}_P)$, the pair $(\chi_{(\beta, \alpha)}, \Phi_{(\beta, \alpha)})$ may not be a $2$-cocycle. However, if $(\beta, \alpha) \in C_{\psi}$, then $(\chi_{(\beta, \alpha)}, \Phi_{(\beta, \alpha)})$ turns out to be a $2$-cocycle. Therefore, we can define a map $\mathcal{W}: C_{\psi} \rightarrow H^{2}_{\ALie}(\mathfrak{g}_P, \mathfrak{h}_Q)$ by 
\begin{align}\label{wells-ab}
\mathcal{W}((\beta, \alpha)) := [ (\chi_{(\beta, \alpha)}, \Phi_{(\beta, \alpha)}) - (\chi, \Phi) ],~\text{for }~~(\beta, \alpha) \in C_{\psi}.
\end{align}
Like the non-abelian case, the Wells map here does not depend on the chosen section. Thus, Theorem \ref{main} in the case of abelian extensions can be rephrased as follows.

\begin{thm}
Let $0 \rightarrow \mathfrak{h}_{Q} \xrightarrow{i}{\mathfrak{e}_{U}} \xrightarrow{p}{\mathfrak{g}_{P}} \rightarrow 0$
be an abelian extension of an averaging Lie algebra $\mathfrak{g}_{P}$ by a representation $\mathfrak{h}_Q$. Then a pair $(\beta, \alpha) \in \Aut(\mathfrak{h}_Q) \times \Aut(\mathfrak{g}_P)$  is inducible if and only if $(\beta,\alpha) \in C_{\psi}$ and $\mathcal{W}((\beta, \alpha))=0$.
\end{thm}
Since the image of the map $\Pi: \Aut_{\mathfrak{h}}(\mathfrak{e}_{U}) \rightarrow \Aut(\mathfrak{h}_Q) \times \Aut(\mathfrak{g}_P)$ lies in $C_{\psi} \subset \Aut(\mathfrak{h}_Q) \times \Aut(\mathfrak{g}_P)$, we have the following Wells short exact sequence for abelian extensions.

\begin{thm}
Let  $0 \rightarrow \mathfrak{h}_{Q} \xrightarrow{i}{\mathfrak{e}_{U}} \xrightarrow{p}{\mathfrak{g}_{P}} \rightarrow 0$
be an abelian extension of an averaging Lie algebra $\mathfrak{g}_{P}$ by a representation $\mathfrak{h}_Q$. Then there is an exact sequence 
\begin{align}\label{exact-ab}
1 \rightarrow  \Aut_{\mathfrak{h}}^{\mathfrak{h},\mathfrak{g}}(\mathfrak{e}_{U}) \xrightarrow{\iota}  \Aut_{\mathfrak{h}}(\mathfrak{e}_{U}) \xrightarrow{\Pi}   C_{\psi} \xrightarrow{\mathcal{W}}  H^{2}_{\nab}(\mathfrak{g}_P, \mathfrak{h}_Q).
\end{align}
\end{thm}

\medskip

An abelian extension $0 \rightarrow \mathfrak{h}_Q \xrightarrow{i} \mathfrak{e}_U \xrightarrow{p} \rightarrow \mathfrak{g}_P \rightarrow 0$ of an averaging Lie algebra $\mathfrak{g}_P$ by a representation $\mathfrak{h}_Q$ is said to be {\bf split} if there exist a section $s: \mathfrak{g} \rightarrow \mathfrak{e}$ which is a morphism of averaging Lie algebras. Then the averaging Lie algebra $\mathfrak{e}_U$ is isomorphic to the averaging Lie algebra $(\mathfrak{g} \oplus \mathfrak{h})_{P \oplus Q}$, where the Lie bracket on $\mathfrak{g} \oplus \mathfrak{h}$ is given by the semi-direct product 
\begin{align*}
    [(x, h), (y, k) ]_\ltimes := ( [x, y]_\mathfrak{g}, \psi_x k - \psi_y h), \text{ for } (x, h), (y, k) \in \mathfrak{g} \oplus \mathfrak{h}.
\end{align*}
Thus, if $(\chi, \Phi)$ is the $2$-cocycle corresponding to the above split abelian extension induced by the section $s$, then
\begin{align*}
    \chi (x, y) = [s(x) , s(y)]_\mathfrak{e} - s [x, y]_\mathfrak{g} = 0 ~~~ \text{ and } ~~~ \Phi(x) = (Us- sP) (x) = 0, \text{ for } x, y \in \mathfrak{g}.
\end{align*}
Therefore, it turns out that the Wells map defined in (\ref{wells-ab}) vanishes identically. Hence, for split abelian extensions, the exact sequence (\ref{exact-ab}) takes the following form
\begin{align}\label{new-exact-ab}
    1 \rightarrow \mathrm{Aut}_\mathfrak{h}^{\mathfrak{h}, \mathfrak{g}} (\mathfrak{e}_U) \xrightarrow{ \iota } \mathrm{Aut}_\mathfrak{h} (\mathfrak{e}_U) \xrightarrow{ \Pi } C_\psi \rightarrow 1.
\end{align}
Note that we can define a group homomorphism $\rho: C_\psi  \rightarrow \mathrm{Aut}_\mathfrak{h} (\mathfrak{e}_U)$ by $\rho ((\beta, \alpha)) (x, h) = \big(  \alpha(x), \beta(h) \big)$, for $(\beta, \alpha) \in C_\psi$ and $(x, h) \in \mathfrak{g} \oplus \mathfrak{h} \cong \mathfrak{e}$. Further, we have $(\Pi \circ \rho) (\beta, \alpha) = \big(  \rho(\beta, \alpha)|_\mathfrak{h}, \overline{   \rho(\beta, \alpha)  }   \big) = (\beta, \alpha)$ as $ \overline{   \rho(\beta, \alpha)  } = p \circ \rho (\beta, \alpha) \circ s = \alpha$. This shows that (\ref{new-exact-ab}) is a split exact sequence in the category of groups. As a consequence, we obtain the following result.

\begin{prop}\label{last-prop}
    Let $0 \rightarrow \mathfrak{h}_Q \xrightarrow{i} \mathfrak{e}_U \xrightarrow{p} \rightarrow \mathfrak{g}_P \rightarrow 0$ be a split abelian extension of an averaging Lie algebra $\mathfrak{g}_P$ by a representation $\mathfrak{h}_Q$. Then as groups,
    \begin{align*}
        \mathrm{Aut}_\mathfrak{h} (\mathrm{e}_U) \cong C_\psi \ltimes \mathrm{Aut}_\mathfrak{h}^{\mathfrak{h}, \mathfrak{g}} (\mathfrak{e}_U),
    \end{align*}
    where the right-hand side is the semi-direct product of groups.
\end{prop}

\medskip

\noindent {\bf Acknowledgements.} The first named author would like to thank Indian Institute of Technology (IIT) Kharagpur for providing the beautiful academic atmosphere where his part of the research has been carried out. The second named author acknowledges the Tata Institute of Fundamental Research (TIFR) Mumbai for their postdoctoral fellowship.

\medskip

\noindent {\bf Data Availability Statement.} Data sharing does not apply to this article as no new data were created or analyzed in this study.

\end{document}